\documentclass{amsart}
\usepackage{amsthm}

\usepackage[latin1]{inputenc}
\usepackage{subfigure}
\usepackage{color}
\usepackage{amsmath}
\usepackage{amsthm}
\usepackage{amstext}
\usepackage{amssymb}
\usepackage{amsfonts}
\usepackage{graphicx}
\usepackage{young}
\usepackage{multicol}
\usepackage{mathrsfs}
\usepackage{stmaryrd}
\usepackage{bbm}
\usepackage[all]{xy}
\usepackage{mathabx}
\usepackage{tikz}
\usepackage{mathtools}
\usepackage{tabularx}
\usepackage{array}

\DeclareMathOperator{\Ext}{\mathrm{Ext}}
\DeclareMathOperator{\Hom}{\mathrm{Hom}}

\theoremstyle{plain}
\theoremstyle{definition}
\newtheorem{theorem}{Theorem}[section]

\newtheorem{remark}[theorem]{Remark}
\newtheorem{lemma}[theorem]{Lemma}
\newtheorem{definition}[theorem]{Definition}

\newtheorem{example}[theorem]{Example}
\newtheorem{proposition}[theorem]{Proposition}
\newtheorem{corollary}[theorem]{Corollary}

\DeclareMathAlphabet{\mathpzc}{OT1}{pzc}{m}{it}

\newcommand{\testleftlong}{\longleftarrow\!\shortmid}

\usepackage{amssymb,amsmath,tabularx,graphicx}
\usepackage{tikz}
\usetikzlibrary[calc,intersections,through,backgrounds,arrows,decorations.pathmorphing]

\def\Bcal{\mathcal{B}}

\def\Tcal{\mathcal{T}}

\def\Zbb{\mathbb{Z}}

\def\ra{\rightarrow}
\def\pr{\prime}
\def\ov{\overline}
\def\un{\underline}
\DeclareMathOperator{\AP}{AP}

\DeclareMathOperator{\Bic}{Bic}
\DeclareMathOperator{\supp}{supp}

\DeclareMathOperator{\op}{op}

\begin{document}

\title{Lattice Properties of Oriented Exchange Graphs and Torsion Classes}
\date{}
\author{Alexander Garver}
\address{Laboratoire de Combinatoire et d'Informatique Math\'ematique,
Universit\'e du Qu\'ebec \`a Montr\'eal, Montr\'eal, QC H3C 3P8, Canada}
\email{alexander.garver@lacim.ca}

\author{Thomas McConville}
\address{Massachusetts Institute of Technology, Cambridge, MA 02139-4307, USA}
\email{thomasmc@mit.edu}

\thanks{The authors were supported by a Research Training Group, RTG grant DMS-1148634.}

\begin{abstract}
The exchange graph of a 2-acyclic quiver is the graph of mutation-equivalent quivers whose edges correspond to mutations. When the quiver admits a nondegenerate Jacobi-finite potential, the exchange graph admits a natural acyclic orientation called the oriented exchange graph, as shown by Br\"ustle and Yang.  The oriented exchange graph is isomorphic to the Hasse diagram of the poset of functorially finite torsion classes of a certain finite dimensional algebra. We prove that lattices of torsion classes are semidistributive lattices, and we use this result to conclude that oriented exchange graphs with finitely many elements are semidistributive lattices.  Furthermore, if the quiver is mutation-equivalent to a type A Dynkin quiver or is an oriented cycle, then the oriented exchange graph is a lattice quotient of a lattice of biclosed subcategories of modules over the cluster-tilted algebra, generalizing Reading's Cambrian lattices in type A. We also apply our results to address a conjecture of Br\"ustle, Dupont, and P\'erotin on the lengths of maximal green sequences.
\end{abstract}

\maketitle

\tableofcontents

\section{Introduction}

The exchange graph defined by a \textbf{2-acyclic} quiver $Q$ (i.e., a quiver with loops or 2-cycles) admits a natural acyclic orientation called the oriented exchange graph of $Q$ (see \cite{by14}). For example, if $Q$ is an orientation of a Dynkin diagram of type $\mathbb{A}$, $\mathbb{D}$, or $\mathbb{E}$, then its oriented exchange graph is a Cambrian lattice of the same type (see \cite{ReadingCamb}).  Maximal directed paths of finite length in the oriented exchange graph, known as \textbf{maximal green sequences} \cite{bdp}, may be used to compute the refined Donaldson-Thomas invariant of $Q$ (see \cite{k12}) introduced by Kontsevich and Soibelman in \cite{KS}.  In the Dynkin case, we may extract combinatorial information about oriented exchange graphs from the Cambrian lattice structure.  The purpose of this work is to uncover similar information about oriented exchange graphs associated to some non-Dynkin quivers of finite type.  We summarize our approach below and distinguish it from other approaches.  Most of the definitions will be given in later sections.

Let $Q$ be a 2-acyclic quiver with $n$ vertices and $\widehat{Q}$ its \textbf{framed quiver}. Let $\overline{Q}$ be a quiver obtained from $\widehat{Q}$ by a finite sequence of mutations. A \textbf{$\textbf{c}$-vector} of $Q$ is a vector whose coordinates describe the direction and multiplicity of arrows connecting some \textbf{mutable vertex} of $\overline{Q}$ to the \textbf{frozen vertices} of $\widehat{Q}$.  The coordinates of a \textbf{c}-vector are known to be either all nonpositive or all nonnegative.  An edge of the exchange graph of $Q$ is oriented in one of two ways depending on whether the entries of the corresponding \textbf{c}-vector are nonpositive or nonnegative; see Section~\ref{Prelim} for a more precise description.

In great generality, it is known that the orientation of an exchange graph associated to a 2-acyclic quiver is known to be acyclic. One approach to proving this fact invokes a surprising connection to representation theory. Suppose $Q$ is 2-acyclic quiver that admits a nondegenerate Jacobi-finite potential (see \cite{dwz1} for more details). Then there is an associative algebra $\Lambda = \Bbbk Q/I$, the \textbf{Jacobian algebra} of $Q$, whose functorially finite torsion classes are in natural bijection with vertices of the oriented exchange graph of $Q$.  Ordering torsion classes by inclusion, the covering relations among functorially finite torsion classes correspond to edges of the exchange graph in such a way that the orientation is preserved. This orientation-preserving bijection was first discovered in \cite{by14}.  Now since an inclusion order on any family of sets is acyclic, it follows that the oriented exchange graph is acyclic.

The set of torsion classes of \textit{any} finite dimensional algebra $\Lambda = \Bbbk Q/I$ is known to form a lattice, and the subposet of functorially finite torsion classes forms a lattice when $\Lambda$ has certain algebraic properties \cite{iyama.reiten.thomas.todorov:latticestrtors} (we will focus on algebras $\Lambda$ having finite lattices of torsion classes so we do not need to distinguish between lattices and complete lattices). We use a formula for the join of two torsion classes shown to us by Hugh Thomas \cite{hthomas} to prove that torsion classes form a semidistributive lattice (see Theorem~\ref{semidis} and the comments around Lemma~\ref{join}). In the situation where every torsion class of $\Lambda$ is functorially finite, which is true if $Q$ is mutation-equivalent to a Dynkin quiver, we conclude that the corresponding oriented exchange graph is semidistributive.

A Cambrian lattice can be constructed either as a special lattice quotient of the weak order of a finite Coxeter group or as an oriented exchange graph of a Dynkin quiver.  We construct a similar lattice quotient description when $Q$ is an oriented cycle or mutation-equivalent to a type $\mathbb{A}$ Dynkin quiver.  First, we define a closure operator on its set of positive \textbf{c}-vectors. We then show that biclosed sets of \textbf{c}-vectors can be interpreted algebraically as what we call biclosed subcategories of the module category of $\Lambda$. After that, we construct a map $\pi_\downarrow$ from biclosed subcategories of $\Lambda$-mod to functorially finite torsion classes of $\Lambda$.

We prove that the set of biclosed sets ordered by inclusion forms a congruence-uniform lattice, which is a stronger property than semidistributivity, and the above map has the structure of a lattice quotient map. The latter implies that the lattice of functorially finite torsion classes is also a congruence-uniform lattice. For any 2-acyclic quiver $Q$ admitting a nondegenerate Jacobi-finite potential, we conjecture that if the functorially finite torsion classes of the Jacobian algebra $\Lambda = \Bbbk Q/I$ form a lattice, then that lattice is congruence-uniform. However, we do not have a proof.

The paper is organized as follows. In Section~\ref{Prelim}, we explain several of the combinatorial and algebraic tools we will use in this paper. We begin by reviewing basic notions such as ice quivers, mutation of ice quivers, exchange graphs of quivers, and oriented exchange graphs of quivers. After that, we review the basic theory of path algebras with relations and their module theory. The class of path algebras with relations that we consider in this paper are cluster-tilted algebras of type $\mathbb{A}$ (i.e., these algebras are defined by quivers that are mutation-equivalent to a type $\mathbb{A}$ Dynkin quiver) and the cluster-tilted algebras defined by a quiver that is an oriented cycle. Cluster-tilted algebras are examples of the Jacobian algebras mentioned earlier. We carefully describe these algebras and their properties in Sections~\ref{clusttilt} and ~\ref{cyclic}. Throughout this paper, when we speak about one of the algebras, we denote it by $\Lambda$. 

In Section~\ref{latticeprops}, we review basic notions related to lattice theory that will be useful to us. Of particular importance to us will be the notions of semidistributive, congruence-uniform, and polygonal lattices.

In Section~\ref{semidisoreg}, we review the concepts of torsion classes and torsion-free classes. Using the fact that the lattice of functorially finite torsion classes of $\Lambda$ is isomorphic to the oriented exchange graph of $Q$, we prove that when $Q$ is mutation-equivalent to a Dynkin quiver its oriented exchange graph is a semidistributive lattice.

In Section~\ref{biclosedsets}, we develop the theory of biclosed sets. We introduce the notion of biclosed sets of acyclic paths in a graph, denoted $\text{Bic}(\text{AP})$. We prove that $\text{Bic}(\text{AP})$ is a semidistributive, congruence-uniform, and polygonal lattice (see Theorem~\ref{bicapprops}). When $Q$ is mutation-equivalent to a path quiver or is an oriented cycle, we can identify the \textbf{c}-vectors of $Q$ with acyclic paths in $Q$. In this way, we can consider the lattice of biclosed sets of \textbf{c}-vectors of $Q$, denoted $\text{Bic}(Q)$, and conclude that this lattice is  semidistributive, congruence-uniform, and polygonal.

In Section~\ref{biclosedsubcat}, we show that $\text{Bic}(Q)$ is isomorphic to what we call the lattice of biclosed subcategories of $\Lambda$-mod, denoted $\mathcal{BIC}(Q)$. Using this categorification, we define maps $\pi_\downarrow$ and $\pi^\uparrow$ on $\mathcal{BIC}(Q)$. Our main theorem is that $\pi_\downarrow: \mathcal{BIC}(Q) \to \text{tors}(\Lambda)$ is a lattice quotient map (see Theorem~\ref{lattquot}). We remark that it is not clear a priori that the image of $\pi_\downarrow$ is contained in $\text{tors}(\Lambda).$ We conclude this section by giving an affirmative answer to a conjecture of Br\"{u}stle, Dupont, and P\'{e}rotin (see \cite[Conjecture 2.22]{bdp}) when $Q$ is mutation-equivalent to a path quiver or is an oriented cycle (see Corollary~\ref{lengths}).

In Section~\ref{aboutpiupanddown}, we prove several important properties of $\pi_\downarrow$ and $\pi^\uparrow$ that are needed for the proof of Theorem~\ref{lattquot}. Much of this section is dedicated to proving that the image of $\pi_\downarrow$ is contained in $\text{tors}(\Lambda)$. A crucial step in this argument is the use of a basis for the equivalence classes of extensions of one indecomposable $\Lambda$-module by another given in \cite{cs14}. 

In Section~\ref{canonical}, we apply our results about biclosed subcategories to classify canonical join- and canonical meet-representations of torsion classes. 

In Section~\ref{additionallemmas}, we record a few necessary results whose statements and proofs do not fit with the exposition in other sections.

In this paper, we only present a lattice quotient description of oriented exchange graphs defined by quivers that are mutation-equivalent to a path quiver or are an oriented cycle. We believe that one needs a more refined notion of biclosed subcategories in order to produce a lattice quotient description of oriented exchange graphs defined by any finite type quiver. 

{\bf Acknowledgements.~}
Alexander Garver thanks Cihan Bahran, Gregg Musiker, Rebecca Patrias, and Hugh Thomas for several helpful conversations. The authors also thank an anonymous referee for carefully commenting on the manuscript.

\section{Preliminaries}\label{Prelim}

\subsection{Quiver mutation}\label{subsec:quivers}
A \textbf{quiver} $Q$ is a directed graph. In other words, $Q$ is a 4-tuple $(Q_0,Q_1,s,t)$, where $Q_0 = [m] := \{1,2, \ldots, m\}$ is a set of \textbf{vertices}, $Q_1$ is a set of \textbf{arrows}, and two functions $s, t:Q_1 \to Q_0$ defined so that for every $\alpha \in Q_1$, we have $s(\alpha) \xrightarrow{\alpha} t(\alpha)$. An \textbf{ice quiver} is a pair $(Q,F)$ with $Q$ a quiver and $F \subset Q_0$ a set of \textbf{frozen vertices} with the additional restriction that any $i,j \in F$ have no arrows of $Q$ connecting them. We refer to the elements of $Q_0\backslash F$ as \textbf{mutable vertices}. By convention, we assume $Q_0\backslash F = [n]$ and $F = [n+1,m] := \{n+1, n+2, \ldots, m\}.$ Any quiver $Q$ can be regarded as an ice quiver by setting $Q = (Q, \emptyset)$.

If a given ice quiver $(Q,F)$ is \textbf{2-acyclic} (i.e., $Q$ has no loops or 2-cycles), we can define a local transformation of $(Q,F)$ called \textbf{mutation}. The {\bf mutation} of an ice quiver $(Q,F)$ at mutable vertex $k$, denoted $\mu_k$, produces a new ice quiver $(\mu_kQ,F)$ by the three step process:

(1) For every $2$-path $i \to k \to j$ in $Q$, adjoin a new arrow $i \to j$.

(2) Reverse the direction of all arrows incident to $k$ in $Q$.

(3) Remove a maximal collection of disjoint 2-cycles in the resulting quiver as well as all of the arrows between two frozen vertices. 

\noindent We show an example of mutation below depicting the mutable (resp., frozen) vertices in black (resp., blue).
\vspace{-.1in}
\[
\begin{array}{c c c c c c c c c}
\raisebox{-.4in}{$(Q,F)$} & \raisebox{-.4in}{=} & {\begin{xy} 0;<1pt,0pt>:<0pt,-1pt>:: 
(0,30) *+{1} ="0",
(40,0) *+{2} ="1",
(80,30) *+{3} ="2",
(40,60) *+{\textcolor{blue}{4}} ="3",
"0", {\ar@<-.5ex>"1"},
"0", {\ar@<.5ex>"1"},
"1", {\ar"2"},
"3", {\ar"2"},
\end{xy}} & \raisebox{-.4in}{$\stackrel{\mu_2}{\longmapsto}$} & {\begin{xy} 0;<1pt,0pt>:<0pt,-1pt>:: 
(0,30) *+{1} ="0",
(40,0) *+{2} ="1",
(80,30) *+{3} ="2",
(40,60) *+{\textcolor{blue}{4}} ="3",
"2", {\ar"1"},
"0", {\ar@<-.5ex>"2"},
"0", {\ar@<.5ex>"2"},
"1", {\ar@<-.5ex>"0"},
"1", {\ar@<.5ex>"0"},
"3", {\ar"2"},
\end{xy}} & \raisebox{-.4in}{=} & \raisebox{-.4in}{$(\mu_2Q,F)$}
\end{array}
\]

The information of an ice quiver can be equivalently described by its (skew-symmetric) \textbf{exchange matrix}. Given $(Q,F),$ we define $B = B_{(Q,F)} = (b_{ij}) \in \mathbb{Z}^{n\times m} := \{n \times m \text{ integer matrices}\}$ by $b_{ij} := \#\{i \stackrel{\alpha}{\to} j \in Q_1\} - \#\{j \stackrel{\alpha}{\to} i \in Q_1\}.$ Furthermore, ice quiver mutation can equivalently be defined  as \textbf{matrix mutation} of the corresponding exchange matrix. Given an exchange matrix $B \in \mathbb{Z}^{n\times m}$, the \textbf{mutation} of $B$ at $k \in [n]$, also denoted $\mu_k$, produces a new exchange matrix $\mu_k(B) = (b^\prime_{ij})$ with entries
\[
b^\prime_{ij} := \left\{\begin{array}{ccl}
-b_{ij} & : & \text{if $i=k$ or $j=k$} \\
b_{ij} + \frac{|b_{ik}|b_{kj}+ b_{ik}|b_{kj}|}{2} & : & \text{otherwise.}
\end{array}\right.
\]
\begin{flushleft}For example, the mutation of the ice quiver above (here $m=4$ and $n=3$) translates into the following matrix mutation. Note that mutation of matrices {(or of ice quivers)} is an involution (i.e., $\mu_k\mu_k(B) = B$). Let Mut($(Q,F)$) denote the collection of ice quivers obtainable from $(Q,F)$ by finitely many mutations.\end{flushleft}
\[
\begin{array}{c c c c c c c c c c}
B_{(Q,F)} & = & \left[\begin{array}{c c c | r}
0 & 2 & 0 & 0 \\
-2 & 0 & 1 & 0\\
0 & -1 & 0 & -1\\
\end{array}\right]
& \stackrel{\mu_2}{\longmapsto} &
\left[\begin{array}{c c c | r}
0 & -2 & 2 & 0 \\
2 & 0 & -1 & 0\\
-2 & 1 & 0 & -1\\
\end{array}\right] 
& = & B_{(\mu_2Q,F)}.
\end{array}
\]

{Given a quiver $Q$, we define its \textbf{framed} (resp., \textbf{coframed}) quiver to be the ice quiver $\widehat{Q}$ (resp., $\widecheck{Q}$) where $\widehat{Q}_0\ (= \widecheck{Q}_0) := Q_0 \sqcup [n+1, 2n]$, $F = [n+1, 2n]$, and $\widehat{Q}_1 := Q_1 \sqcup \{i \to n+i: i \in [n]\}$ (resp., $\widecheck{Q}_1 := Q_1 \sqcup \{n+i \to i: i \in [n]\}$).}  Now given $\widehat{Q}$ we define the \textbf{exchange tree} of $\widehat{Q}$, denoted $ET(\widehat{Q})$, to be the (a priori infinite) graph whose vertex set is Mut$(\widehat{Q})$ with an edge between two vertices if and only if the quivers corresponding to those vertices are obtained from each other by a single mutation. Similarly, define the \textbf{exchange graph} of $\widehat{Q}$, denoted $EG(\widehat{Q})$, to be the quotient of $ET(\widehat{Q})$ where two vertices are identified if and only if there is a \textbf{frozen isomorphism} of the corresponding quivers (i.e., an isomorphism of quivers that fixes the frozen vertices). Such an isomorphism is equivalent to a simultaneous permutation of the rows and first $n$ columns of the corresponding exchange matrices.

In this paper, we focus our attention on \textbf{type $\mathbb{A}$ quivers} (i.e., quivers $R \in \text{Mut}(1 \leftarrow 2 \leftarrow \cdots \leftarrow n)$ for some positive integer $n$). We will use the following classification due to Buan and Vatne in our study of type $\mathbb{A}$ quivers. We write $\underline{Q}$ for the undirected multigraph obtained by forgetting the orientations of the arrows of $Q$. By a \textbf{circuit} of an undirected graph $\un{Q}$ we mean a minimal subset of $Q_1$ that form a cycle in $\un{Q}$.

\begin{lemma}\cite[Proposition 2.4]{bv08}\label{BV}
A quiver $Q$ is of type $\mathbb{A}$ if and only if $Q$ satisfies the following:
\begin{itemize}
     \item[$i)$] any circuit of $\un{Q}$ is a directed 3-cycle of $Q$; 
     \item[$ii)$] any vertex has at most four neighbors; 
     \item[$iii)$] if a vertex has four neighbors, then two of its adjacent arrows belong to one 3-cycle, and the other two belong to another 3-cycle;
     \item[$iv)$] if a vertex has exactly three neighbors, then two of its adjacent arrows belong to a 3-cycle, and the third arrow does not belong to any 3-cycle. 
     \end{itemize} 
\end{lemma}

\subsection{Oriented exchange graphs} In this brief section, we recall the definitions of \textbf{c}-vectors and their sign-coherence property. We use these notions to explain how to orient the edges of $EG(\widehat{Q})$ for a given quiver $Q$ to obtain the \textbf{oriented exchange graph} of $Q$, denoted $\overrightarrow{EG}(\widehat{Q})$. Oriented exchange graphs were introduced in \cite{bdp} and were shown to be isomorphic to many important partially-ordered sets in representation theory in \cite{by14}.

Given ${Q}$, we define the \textbf{c}-\textbf{matrix} ${C} = {C}_R$ (resp., $\overline{C} = \overline{C}_R$) of $R \in ET(\widehat{Q})$ (resp., $R \in EG(\widehat{Q})$)  to be the submatrix of $B_R$ where $C := (b_{ij})_{i \in [n], j \in [n+1, 2n]}$ (resp., $\overline{C} := (b_{ij})_{i \in [n], j \in [n+1,2n]}$). We let \textbf{c}-mat($Q$) $:= \{\overline{C}_R: R \in EG(\widehat{Q})\}$. By definition, $B_R$ (resp., $\overline{C}$) is only defined up to simultaneous permutations of its rows and its first $n$ columns (resp., up to permutations of its rows) for any $R \in EG(\widehat{Q})$.

A row vector of a \textbf{c}-matrix, $\textbf{c}_i$, is known as a \textbf{c}-\textbf{vector}. We will denote the set of \textbf{c}-vectors of $Q$ by \textbf{c}-vec($Q$). The celebrated theorem of Derksen, Weyman, and Zelevinsky \cite[Theorem 1.7]{dwz10}, known as the {sign-coherence} of $\textbf{c}$-vectors, states that for any $R \in ET(\widehat{Q})$ and any $i \in [n]$ the \textbf{c}-vector $\textbf{c}_i$ is a nonzero element of $\mathbb{Z}_{\ge 0}^n$ or $\mathbb{Z}_{\le0}^n$. We say a \textbf{c}-vector is \textbf{positive} in the former case and \textbf{negative} in the latter case. A mutable vertex $i$ of an ice quiver $(R,F) \in \text{Mut}(\widehat{Q})$ is said to be {\bf green} (resp., {\bf red}) if all arrows of $(R,F)$ connecting an element of $F$ and $i$ point away from (resp., towards) $i$.  Note that all vertices of $\widehat{Q}$ are green and all vertices of $\widecheck{Q}$ are red. We use the notion of green and red vertices to orient the edges of $EG(\widehat{Q})$ to obtain the oriented exchange graph of $Q$.

\begin{definition}\cite{bdp}
Let $Q$ be a quiver. The \textbf{oriented exchange graph} of $Q$, denoted $\overrightarrow{EG}(\widehat{Q})$, is the directed graph whose underlying unoriented graph is $EG(\widehat{Q})$ with its edges oriented as follows. If $(R^1,F)$ and $(R^2,F)$ are connected by an edge in $EG(\widehat{Q})$, then there is a directed edge $(R^1,F) \longrightarrow (R^2,F)$ if $R^2 = \mu_kR^1$ where $k \in R^1_0$ is green, otherwise there is a directed edge $(R^1, F) \longleftarrow (R^2,F)$. 

We define a \textbf{maximal green sequence} of $Q$, denoted $\textbf{i} = (i_1,\ldots, i_k)$, to be a finite sequence of mutable vertices of $\widehat{Q}$ where

$\begin{array}{rll}
a) & i_j \text{ is green in $\mu_{i_{j-1}}\circ \cdots \circ \mu_{i_1}(\widehat{Q})$ for each } j \in [k], \text{ and} \\
b) & \mu_{i_{k}}\circ \cdots \circ \mu_{i_1}(\widehat{Q}) \text{ has only red vertices.}
\end{array}$

Let $\text{green}(Q)$ denote the set of maximal green sequences of $Q$. By definition, maximal green sequences of $Q$ are in bijection with the finite length maximal directed paths in $\overrightarrow{EG}(\widehat{Q})$. Additionally, we define $\ell en(\textbf{i})= k$ to be the \textbf{length} of maximal green sequence $\textbf{i} = (i_1,\ldots, i_k)$.
\end{definition}

\begin{example} Let $Q = 1\longrightarrow 2$. Below we show $\overrightarrow{EG}(\widehat{Q})$ and all of the \textbf{c}-matrices in $\textbf{c}\text{-mat(}Q\text{)}.$ Additionally, we note that $\textbf{i}_1 = (1,2)$ and $\textbf{i}_2 = (2,1,2)$ are the two maximal green sequences of $Q$.
\begin{figure}[h]
$$\begin{array}{rccccl}
\raisebox{.85in}{$\overrightarrow{EG}(\widehat{Q})$} & \raisebox{.85in}{=} &\includegraphics[scale=1]{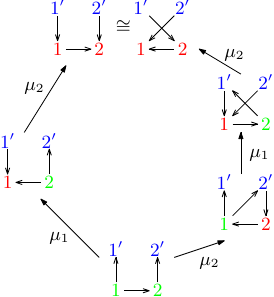} & \raisebox{.85in}{=} & \includegraphics[scale=1]{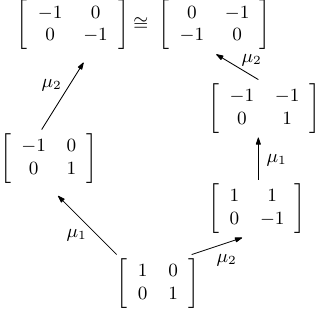}
\end{array}$$
\caption{The oriented exchange graph of $Q=1\rightarrow 2$.}
\label{oregA2}
\end{figure}
\end{example}

\subsection{Path algebras and quiver representations} Following \cite{ass06}, let $Q$ be a given quiver. We define a \textbf{path} of \textbf{length} $\ell \ge 1$ to be an expression $\alpha_1\alpha_2\cdots\alpha_\ell$ where $\alpha_i \in Q_1$ for all $i \in [\ell]$ and $s(\alpha_i) = t(\alpha_{i+1})$ for all $i \in [\ell-1]$. We may visualize such a path in the following way $$\begin{xy} 0;<1pt,0pt>:<0pt,-1pt>:: 
(0,0) *+{\cdot} ="0",
(25,0) *+{\cdot} ="1",
(50,0) *+{\cdot} ="2",
(75,0) *+{\cdot} ="3",
(100,0) *+{\cdots} ="4",
(125,0) *+{\cdot} ="5",
(150,0) *+{\cdot} ="6",
(175,0) *+{\cdot} ="7",
"1", {\ar_{\alpha_1}"0"},
"2", {\ar_{\alpha_2}"1"},
"3", {\ar"2"},
"6", {\ar"5"},
"7", {\ar_{\alpha_\ell}"6"},
\end{xy}.$$ Furthermore, the \textbf{source} (resp., \textbf{target}) of the path $\alpha_1\alpha_2\cdots\alpha_\ell$ is $s(\alpha_\ell)$ (resp., $t(\alpha_1)$). Let $Q_\ell$ denote the set of all paths in $Q$ of length $\ell$. We also associate to each vertex $i \in Q_0$ a path of length $\ell = 0$, denoted $\varepsilon_i$, that we will refer to as the \textbf{lazy path} at $i$. 

\begin{definition}
The \textbf{path algebra} of $Q$ is the $\Bbbk$-algebra generated as a vector space by all paths of length $\ell \ge 0$. Throughout this paper, we assume that $\Bbbk$ is algebraically closed. The multiplication of two paths is induced by `concatenation of paths'. We refer the reader to \cite[Chapter II]{ass06} for more details. We will denote the path algebra of $Q$ by $\Bbbk Q$. Note also that as $\Bbbk$-vector spaces we have $$\Bbbk Q = \bigoplus_{\ell = 0}^\infty \Bbbk Q_\ell$$ where $\Bbbk Q_\ell$ is the $\Bbbk$-vector space of all paths of length $\ell$.
\end{definition}

In this paper, we study certain quivers $Q$ which have \textbf{oriented cycles}. We say a path of length $\ell \ge 0$ $\alpha_1 \cdots \alpha_\ell \in Q_\ell$ is an \textbf{oriented cycle} if $t(\alpha_1) = s(\alpha_\ell)$. If a quiver $Q$ possesses any oriented cycles of length $\ell \ge 1$, we see that $\Bbbk Q$ is infinite dimensional. If $Q$ has no oriented cycles of positive length, we say that $Q$ is \textbf{acyclic}.

In order to avoid studying infinite dimensional algebras, we will add relations to path algebras whose quivers contain oriented cycles in such a way that we obtain finite dimensional quotients of path algebras. The relations we add are those coming from an \textbf{admissible} ideal $I$ of $\Bbbk Q$ meaning that there exists $r \ge 2$ such that $$\bigoplus_{\ell = r}^\infty \Bbbk Q_\ell \subset I \subset \bigoplus_{\ell = 2}^\infty \Bbbk Q_\ell.$$ If $I$ is an admissible ideal of $\Bbbk Q$, we say that $(Q,I)$ is a \textbf{bound quiver} and that $\Bbbk Q/I$ is a \textbf{bound quiver algebra}.

In this paper, we study modules over a bound quiver algebra $\Bbbk Q/I$ by studying certain quiver representations of $Q$ that are ``compatible" with the relations coming from $I$. A \textbf{representation} $V = ((V_i)_{i \in Q_0}, (\varphi_\alpha)_{\alpha \in Q_1})$ of a quiver $Q$  is an assignment of a $\Bbbk$-vector space $V_i$ to each vertex $i$ and a $\Bbbk$-linear map $\varphi_\alpha: V_{s(\alpha)} \rightarrow V_{t(\alpha)}$ to each arrow $\alpha \in Q_1$. If $\rho \in \Bbbk Q$, it can be expressed as $$\rho = \sum_{i = 1}^m c_i\alpha^{(i)}_{1}\cdots \alpha^{(i)}_{k_i}$$ where $c_i \in \Bbbk$ and $\alpha^{(i)}_{1}\cdots \alpha^{(i)}_{k_i}\in Q_{k_i}$ so when considering a representation $V$ of $Q$, we define $$\varphi_\rho := \sum_{i = 1}^m c_i\varphi_{\alpha^{(i)}_{1}}\cdots \varphi_{\alpha^{(i)}_{k_i}}.$$  If we have a bound quiver $(Q,I)$, we define a representation of $Q$ \textbf{bound by} $I$ to be a representation of $Q$ where $\varphi_\rho = 0$ for any $\rho \in I.$ We say a representation of $Q$ bound by $I$ is \textbf{finite dimensional} if $\dim_\Bbbk V_i < \infty$ for all $i \in Q_0.$ It turns out that $\Bbbk Q/I$-mod is equivalent to the category of finite dimensional representations of $Q$ bound by $I$. In the sequel, we use this fact without mentioning it further. Additionally, the \textbf{dimension vector} of $V \in \Bbbk Q/I$-mod is the vector $\underline{\dim}(V):=(\dim_\Bbbk V_i)_{i\in Q_0}$ and the \textbf{dimension} of $V$ is defined as $\dim_\Bbbk(V) = \sum_{i \in Q_0} \dim_\Bbbk V_i$. The \textbf{support} of $V \in \Bbbk Q/I$-mod is the set $\text{supp}(V) := \{i\in Q_0 : V_i \neq 0\}$.

\subsection{Cluster-tilted algebras and \textbf{c}-vectors}\label{clusttilt}

In this section, we review the definition of cluster-tilted algebras \cite{bmr1} and their connections with \textbf{c}-vectors \cite{chavez}. As we will focus on cluster-tilted algebras of type $\mathbb{A}$, we recall a description of these algebras as bound quiver algebras and a useful classification of the indecomposable modules over these algebras.

To define cluster-tilted algebras, we need to recall the definition of the cluster category of an acyclic quiver $Q$, which was introduced in \cite{bmrrt}. Let $Q$ be an acyclic quiver.  Let $\mathcal{D} := \mathcal{D}^b(\Bbbk Q\text{-mod})$ denote the bounded derived category of $\Bbbk Q$-mod. Let $\tau : \mathcal{D} \to \mathcal{D}$ denote the \textbf{Auslander-Reiten translation}, and let $[1]: \mathcal{D} \to \mathcal{D}$ denote the shift functor. We define the \textbf{cluster category} of $Q$, denoted $\mathcal{C}_Q$, to be the orbit category $\mathcal{D}/\tau^{-1}[1].$ The objects of $\mathcal{C}_Q$ are $\tau^{-1}[1]$-orbits of objects in $\mathcal{D}$, denoted $\overline{M} := ((\tau^{-1}[1])^i M)_{i \in \mathbb{Z}}$, where $M \in \mathcal{D}$. The morphisms between $\overline{M}, \overline{N} \in \mathcal{C}_Q$ are given by $$\text{Hom}_{\mathcal{C}_Q}\left(\overline{M},\overline{N}\right) := \bigoplus_{i \in \mathbb{Z}}\text{Hom}_{\mathcal{D}}(M, (\tau^{-1}[1])^iN).$$ 

Cluster categories were invented to provide an additive categorification of cluster algebras. We will not discuss cluster algebras in this paper, but we remark that \textbf{cluster-tilting objects} in $\mathcal{C}_Q$, which we now define, are in bijection with the clusters of the cluster algebra $\mathcal{A}_Q$ associated to $Q$.

\begin{definition}(\cite{bmrrt})
We say $T \in \mathcal{C}_Q$ is a \textbf{cluster-tilting object} if\\
(1) $\text{Ext}_{\mathcal{C}_Q}^1(T,T) = 0$ and \\
(2) $T = \bigoplus_{i = 1}^nT_i$ where $\{T_i\}_{i=1}^n$ is a maximal collection of pairwise non-isomorphic indecomposable objects in $\mathcal{C}_Q$.
\end{definition}

Now we define a \textbf{cluster-tilted algebra} to be the endomorphism algebra $\Lambda := \text{End}_{\mathcal{C}_Q}(T)^{\text{op}}$ where $T = \bigoplus_{i=1}^nT_i$ is a cluster-tilting object in $\mathcal{C}_Q$. If $Q$ is a \textbf{Dynkin quiver} (i.e., the underlying graph of $Q$ is a Dynkin graph $\Delta \in \{\mathbb{A}_n, \mathbb{D}_m, \mathbb{E}_6, \mathbb{E}_7, \mathbb{E}_8\}$ with $n \ge 1$ and $m \ge 4$), we say that $\Lambda$ is of \textbf{type} $\Delta$ or of \textbf{Dynkin type}. It follows from \cite[Corollary 2.3]{bmr1} that $\Lambda$ is representation-finite if and only if $\Bbbk Q$ is representation-finite. Thus a cluster-tilted algebra $\Lambda$ is representation-finite if and only if $Q$ is of Dynkin type.

Cluster-tilted algebras of Dynkin type can be described explicitly as bound quiver algebras (see \cite{bmr2}). In the sequel, we use the following description of cluster-tilted algebras of type $\mathbb{A}$ as bound quiver algebras. The following result appeared in \cite{ccs1} and was generalized in \cite{ccs2} and \cite{bmr2}.

\begin{lemma}\label{typeA}
A cluster-tilted algebra $\Lambda$ is of type $\mathbb{A}$ if and only if $\Lambda \cong \Bbbk Q/I$ where $Q$ is a type $\mathbb{A}$ quiver and $I$ is generated by all 2-paths $\alpha\beta \in Q_2$ where $\alpha$ and $\beta$ are two of the arrows of a 3-cycle of $Q$.
\end{lemma}

Using Lemma~\ref{typeA} and the language of string modules, we can explicitly parameterize the indecomposable modules of a type $\mathbb{A}$ cluster-tilted algebra. A \textbf{string algebra} $\Lambda = \Bbbk Q/I$ is a bound quiver algebra where: 

\begin{itemize}
\item[i)] for each vertex $i$ of $Q$ at most two arrows of $Q$ start at $i$ and at most two arrows of $Q$ end at $i$;
\item[ii)] for each arrow $\beta \in Q_1$ there is at most one arrow $\alpha \in Q_1$ and at most one arrow $\gamma \in Q_1$ such that $\alpha\beta \not \in I$ and $\beta\gamma \not \in I$;
\item[iii)] the ideal is generated by paths of length at least two.
\end{itemize}

\noindent A \textbf{string} in $\Lambda$ is a sequence $$w = x_1 \stackrel{\alpha_1}{\longleftrightarrow} x_2 \stackrel{\alpha_2}{\longleftrightarrow} \cdots \stackrel{\alpha_{m}}{\longleftrightarrow} x_{m+1} $$
where $x_i \in Q_0$ for all $i \in [m+1]$, $\alpha_i \in Q_1$ for all $i \in [m]$, each $\alpha_i$ \textbf{connects} $x_i$ and $x_{i+1}$ (i.e., either $s(\alpha_i) = x_i$ and $t(\alpha_i) = x_{i+1}$ or $s(\alpha_i) = x_{i+1}$ and $t(\alpha_i) = x_i$), and $w$ contains no \textbf{substrings} of $w$ of the following forms:

\begin{itemize}
\item[i)] $x \stackrel{\beta}{\longrightarrow} y \stackrel{\beta}{\longleftarrow} x \text{ or } x \stackrel{\beta}{\longleftarrow} y \stackrel{\beta}{\longrightarrow} x$, and 
\item[ii)] $x_{i_1} \stackrel{\beta_1}{\longrightarrow} x_{i_2} \cdots x_{i_{s}} \stackrel{\beta_{s}}{\longrightarrow} x_{i_{s+1}}$ or  $x_{i_1} \stackrel{\gamma_1}{\longleftarrow} x_{i_2} \cdots x_{i_{s}} \stackrel{\gamma_{s}}{\longleftarrow} x_{i_{s+1}}$ where $\beta_s \cdots \beta_1, \gamma_1\cdots \gamma_s \in I$.
\end{itemize}

\noindent In other words, $w$ is an irredundant walk in $Q$ that avoids the relations imposed by $I$. We can translate $w$ into a word $\alpha_1^{\epsilon_1}\cdots \alpha_m^{\epsilon_m}$ with $\epsilon_i \in \{\pm1\}$ for all $i \in [m]$ where $\epsilon_i = 1$ (resp., $\epsilon_i = -1$) if $x_i \stackrel{\alpha_i}{\longleftrightarrow} x_{i+1} = x_i \stackrel{\alpha_i}{\longrightarrow} x_{i+1}$ (resp., $x_i \stackrel{\alpha_i}{\longleftrightarrow} x_{i+1} = x_i \stackrel{\alpha_i}{\longleftarrow} x_{i+1}$). Using this notation, as in \cite{cs14}, we consider strings up to inverses.

Let $w$ be a string in $\Lambda$. The \textbf{string module} defined by $w$ is the bound quiver representation $ M(w) := ((V_i)_{i \in Q_0}, (\varphi_\alpha)_{\alpha\in Q_1})$ with
$$\begin{array}{cccccccccccc}
V_i & := & \left\{\begin{array}{lcl} \displaystyle\bigoplus_{j: x_j = i}\Bbbk x_j &: & \text{if } i = x_j \text{ for some } j \in [m+1]\\ 0 & : & \text{otherwise} \end{array}\right.
\end{array}$$ for each $i \in Q_0$ and 
$$\begin{array}{cccccccccccc}
\varphi_\alpha(x_k) & := & \left\{\begin{array}{lcl} x_{k-1} &: & \text{if } \alpha = \alpha_{k-1} \text{ and } \epsilon_k = -1\\ x_{k+1} &: & \text{if } \alpha = \alpha_{k} \text{ and } \epsilon_k = 1\\  0 & : & \text{otherwise} \end{array}\right.
\end{array}$$for each $\alpha \in Q_0$. One observes that $M(w)\cong M(w^{-1})$.

If $\Bbbk Q/I$ is a representation-finite string algebra, it follows from \cite{wald1985tame} that the set of isomorphism classes of indecomposable $\Bbbk Q/I$-modules, denoted $\text{ind}(\Bbbk Q/I\text{-mod})$, consists of exactly the string modules over $\Bbbk Q/I$. Furthermore, if $M(w) = ((V_i)_{i \in Q_0}, (\varphi_\alpha)_{\alpha\in Q_1})$ is a string module over $\Bbbk Q/I$, a cluster-tilted algebra of type $\mathbb{A}$, then the relations in $\Bbbk Q/I$ require that $\dim_\Bbbk V_i \le 1$ for all $i \in Q_0.$

\begin{example}\label{stringsA3}
Let $Q$ denote the type $\mathbb{A}$ quiver shown below and $I = \langle \beta\alpha, \gamma\beta, \alpha\gamma\rangle$. Then $\Bbbk Q/I = \Bbbk Q/\langle \beta\alpha, \gamma\beta, \alpha\gamma \rangle$. $$\begin{array}{ccccccccccccccc}
\raisebox{-.4in}{$Q$} & \raisebox{-.4in}{=} & \raisebox{-.2in}{$\begin{xy} 0;<1pt,0pt>:<0pt,-1pt>:: 
(0,20) *+{1} ="0",
(20,0) *+{2} ="1",
(40,20) *+{3} ="2",
"0", {\ar^{\alpha}"1"},
"1", {\ar^{\beta}"2"},
"2", {\ar^{\gamma}"0"},
\end{xy}$}  \end{array}$$

\noindent The algebra $\Bbbk Q/I$ has the following indecomposable (string) modules. $$\begin{array}{rcccrcccrccccccc}
\raisebox{-.4in}{$M(1)$} & \raisebox{-.4in}{=} & \raisebox{-.2in}{$\begin{xy} 0;<1pt,0pt>:<0pt,-1pt>:: 
(0,20) *+{\Bbbk} ="0",
(20,0) *+{0} ="1",
(40,20) *+{0} ="2",
"0", {\ar^{0}"1"},
"1", {\ar^{0}"2"},
"2", {\ar^{0}"0"},
\end{xy}$} & & \raisebox{-.4in}{$M(2)$} & \raisebox{-.4in}{=} & \raisebox{-.2in}{$\begin{xy} 0;<1pt,0pt>:<0pt,-1pt>:: 
(0,20) *+{0} ="0",
(20,0) *+{\Bbbk} ="1",
(40,20) *+{0} ="2",
"0", {\ar^{0}"1"},
"1", {\ar^{0}"2"},
"2", {\ar^{0}"0"},
\end{xy}$} & & \raisebox{-.4in}{$M(3)$} & \raisebox{-.4in}{=} & \raisebox{-.2in}{$\begin{xy} 0;<1pt,0pt>:<0pt,-1pt>:: 
(0,20) *+{0} ="0",
(20,0) *+{0} ="1",
(40,20) *+{\Bbbk} ="2",
"0", {\ar^{0}"1"},
"1", {\ar^{0}"2"},
"2", {\ar^{0}"0"},
\end{xy}$}\\
\raisebox{-.4in}{$M(1 \stackrel{\alpha}{\longrightarrow} 2)$} & \raisebox{-.4in}{=} & \raisebox{-.2in}{$\begin{xy} 0;<1pt,0pt>:<0pt,-1pt>:: 
(0,20) *+{\Bbbk} ="0",
(20,0) *+{\Bbbk} ="1",
(40,20) *+{0} ="2",
"0", {\ar^{1}"1"},
"1", {\ar^{0}"2"},
"2", {\ar^{0}"0"},
\end{xy}$} & & \raisebox{-.4in}{$M(2 \stackrel{\beta}{\longrightarrow} 3)$} & \raisebox{-.4in}{=} & \raisebox{-.2in}{$\begin{xy} 0;<1pt,0pt>:<0pt,-1pt>:: 
(0,20) *+{0} ="0",
(20,0) *+{\Bbbk} ="1",
(40,20) *+{\Bbbk} ="2",
"0", {\ar^{0}"1"},
"1", {\ar^{1}"2"},
"2", {\ar^{0}"0"},
\end{xy}$} & & \raisebox{-.4in}{$M(3 \stackrel{\gamma}{\longrightarrow} 1)$} & \raisebox{-.4in}{=} & \raisebox{-.2in}{$\begin{xy} 0;<1pt,0pt>:<0pt,-1pt>:: 
(0,20) *+{\Bbbk} ="0",
(20,0) *+{0} ="1",
(40,20) *+{\Bbbk} ="2",
"0", {\ar^{0}"1"},
"1", {\ar^{0}"2"},
"2", {\ar^{1}"0"},
\end{xy}$}
\end{array}$$
\end{example}

The final result that we present in this section allows us to connect the representation theory of cluster-tilted algebras $\Lambda = \Bbbk Q/I$ of finite representation type with the combinatorics of the \textbf{c}-vectors of $Q$. 

\begin{proposition}\cite[Theorem 6]{chavez}\label{c-vecbij}
Let $Q$ be a quiver that is mutation-equivalent to a Dynkin quiver and let $\Lambda \cong \Bbbk Q/I$ denote the cluster-tilted algebra associated to $Q$. Then we have a bijection $$\begin{array}{rcl}
\text{ind}(\Bbbk Q/I\text{-mod}) & \longrightarrow & \{\textbf{c} \in \textbf{c}\text{-vec(}Q\text{)}: \ \textbf{c} \text{ is positive} \}\\
V & \longmapsto & \underline{\dim}(V).
\end{array}$$
\end{proposition}

As every \textbf{c}-vector is either positive or negative, we may deduce that the set of all \textbf{c}-vectors of $Q$ is
$$\textbf{c}\text{-vec(}Q\text{)} = \{\underline{\dim}(V): V \in \text{ind}(\Bbbk Q/I\text{-mod})\} \bigsqcup \{-\underline{\dim}(V): V \in \text{ind}(\Bbbk Q/I\text{-mod})\}.$$

\begin{example}
Let $Q$ be the quiver appearing in Example~\ref{stringsA3}. By Proposition~\ref{c-vecbij}, we have that $$\textbf{c}\text{-vec}(Q) = \{\pm(1,0,0), \pm(0,1,0), \pm(0,0,1), \pm(1,1,0), \pm(0,1,1), \pm(1,0,1)\}.$$
\end{example}

\subsection{Cyclic quivers}\label{cyclic}

In this section, we describe the second family of bound quiver algebras that we will study. To begin, let $Q(n)$ denote the quiver with $Q(n)_0 := [n]$ and $Q(n)_1 := \{i \to i+1: i \in [n-1]\}\sqcup \{n\to1\}$. For example, when $n = 4$ we have
$$\begin{array}{ccc}
\raisebox{-.25in}{Q(4)} & \raisebox{-.25in}{=} & 
\begin{xy} 0;<1pt,0pt>:<0pt,-1pt>:: 
(0,0) *+{1} ="0",
(0,40) *+{4} ="1",
(40,40) *+{3.} ="2",
(40,0) *+{2} ="3",
"0", {\ar"3"},
"2", {\ar"1"},
"3", {\ar"2"},
"1", {\ar"0"},
\end{xy}
\end{array}$$
As discussed in \cite[Proposition 2.6, Proposition 2.7]{bmr2}, the algebra $$\Lambda = \Bbbk Q(n)/\langle \alpha_1\cdots \alpha_{n-1}: \ \alpha_i \in Q(n)_1\rangle$$ is cluster-tilted of type $\mathbb{D}_n$. As such, $\Lambda$ is representation-finite. Furthermore, one observes that $\Lambda$ is a string algebra and thus the indecomposables $\Lambda$-modules are string modules. One can verify the following lemma.

\begin{lemma}\label{homQn}
Let $w_2 = x_1^{(2)} \leftarrow \cdots \leftarrow x_{k_2}^{(2)}$ and $w_1 = x_1^{(1)} \leftarrow \cdots \leftarrow x_{k_1}^{(1)}$ be strings in $\Lambda.$ Then $\Hom(M(w_2), M(w_1)) \neq 0$ if and only if $x_{k_2}^{(2)} = x_i^{(1)}$ for some $i \in [k_1]$ and $x_1^{(2)} \not \in \{x_2^{(1)}, \ldots, x_{i-1}^{(1)}\}.$ Furthermore, if $\Hom(M(w_2), M(w_1)) \neq 0$, then $\{\theta\}$ is a $\Bbbk$-basis for $\Hom_{\Lambda}(M(w_2),M(w_1))$ where

$$\begin{array}{rcl}
\theta_{x_j^{(1)}} & = & \left\{\begin{array}{rcl}1 & : & \text{ if } j \in [i] \\ 0 & : & \text{ otherwise.} \end{array}\right.
\end{array}$$
\end{lemma}

It will be convenient to introduce an alternative notation for the indecomposable $\Lambda$-modules. Let $X(i,j)$ where $i \in [n]$ and $j \in [n-1]$ denote the unique indecomposable $\Lambda$-module containing $M(i)$  and whose length is $j$. For example, $X(n,i) = M(n \leftarrow \cdots \leftarrow n-i+1)$.

\begin{example}
Let $\Lambda = \Bbbk Q(4)/\langle \alpha_1\alpha_2\alpha_3: \ \alpha_i \in Q(4)_1\rangle.$ Then the \textbf{Auslander-Reiten quiver} of $\Lambda$ is
$$\begin{array}{rcl}
\raisebox{-.4in}{$\Gamma(\Lambda\text{-mod})$} & \raisebox{-.4in}{=} & \begin{xy} 0;<1pt,0pt>:<0pt,-1pt>:: 
(0,60) *+{X(4,1)} ="0",
(60,60) *+{X(3,1)} ="1",
(120,60) *+{X(2,1)} ="2",
(180,60) *+{X(1,1)} ="3",
(240,60) *+{X(4,1)} ="4",
(30,30) *+{X(4,2)} ="5",
(60,0) *+{X(4,3)} ="6",
(0,0) *+{X(1,3)} ="7",
(90,30) *+{X(3,2)} ="8",
(120,0) *+{X(3,3)} ="9",
(150,30) *+{X(2,2)} ="10",
(210,30) *+{X(1,2)} ="11",
(180,0) *+{X(2,3)} ="12",
(240,0) *+{X(1,3)} ="13",
"1", {\ar^\tau@{-->}"0"},
"0", {\ar@{^{(}->}"5"},
"2", {\ar^\tau@{-->}"1"},
"5", {\ar@{->>}"1"},
"1", {\ar@{^{(}->}"8"},
"3", {\ar^\tau@{-->}"2"},
"8", {\ar@{->>}"2"},
"2", {\ar@{^{(}->}"10"},
"10", {\ar@{->>}"3"},
"3", {\ar@{^{(}->}"11"},
"11", {\ar@{->>}"4"},
"5", {\ar@{^{(}->}"6"},
"7", {\ar@{->>}"5"},
"6", {\ar@{->>}"8"},
"8", {\ar@{^{(}->}"9"},
"9", {\ar@{->>}"10"},
"10", {\ar@{^{(}->}"12"},
"12", {\ar@{->>}"11"},
"11", {\ar@{^{(}->}"13"},
"4", {\ar^\tau@{-->}"3"},
"8", {\ar^\tau@{-->}"5"},
"10", {\ar^\tau@{-->}"8"},
"11", {\ar^\tau@{-->}"10"},
\end{xy}.
\end{array}$$
\end{example} 

\begin{remark}
For any quiver $Q(n)$, the Auslander-Reiten quiver of the corresponding cluster-tilted algebra may be embedded on a cylinder. In general, the irreducible morphisms between indecomposable $\Lambda$-modules are exactly those of the form $$X(i,j) \hookrightarrow X(i,j+1) \ \ \ \ \ \text{and} \ \ \ \ \ X(i,j) \twoheadrightarrow X(i-1,j-1).$$

\noindent Also, if $X(i,j) \in \text{ind}(\Lambda\text{-mod})$ and $j \in [n-2]$, then $\tau^{\pm 1} X(i,j) = X(i\pm 1,j)$ where we agree that $\tau X(n,j) = X(1,j)$ and $\tau^{-1} X(1,j) = X(n,j).$  Roughly speaking, $\tau$ acts on non-projective modules by rotation of dimension vectors.  The modules $\{X(i,n-1)\}_{i \in [n]}$ are both the indecomposable projective and indecomposable injective modules so $\Lambda$ is self-injective. Thus $\tau^{\pm 1}X(i,n-1) = 0$ for any $i\in [n].$
\end{remark}

We conclude this section by classifying extensions of indecomposable modules $M(w_1), M(w_2) \in \text{ind}(\Lambda\text{-mod})$ where $\Lambda=\Bbbk Q(n)/\langle \alpha_1\cdots \alpha_{n-1}: \ \alpha_i \in Q(n)_1\rangle$ (i.e., extensions of the form $0\to M(w_2) \to Z \to M(w_1) \to 0$ where $Z \in \Lambda\text{-mod}.$) This classification will be an important tool in the proofs of our main results. The first Lemma we present can be easily verified by considering the structure of the Auslander-Reiten quiver of $\Lambda.$ Recall that $$\underline{\Hom}_\Lambda(M,N) := \{f \in \Hom_\Lambda(M,N): \ f \text{ factors through a projective $\Lambda$-module}\}.$$

\begin{lemma}\label{homstabhom}
If $i\in\Zbb/n\Zbb$ and $1\leq j<n-1$, then
$$\begin{array}{rl} a) & \{X\in\text{ind}(\Lambda\text{-mod}):\ {\Hom}_{\Lambda}(X(i,j),X)\neq 0\}=\left\{X(s,t) \in \text{ind}(\Lambda\text{-mod}):\ \begin{array}{c}s\in[i-j+1,i]_n,\\ j-d(i,s)\leq t\leq n-1\end{array}\right\}
\ \text{and}\\
b) & \{X\in\text{ind}(\Lambda\text{-mod}):\ \underline{\Hom}_{\Lambda}(X(i,j),X)\neq 0\}=\left\{X(s,t) \in \text{ind}(\Lambda\text{-mod}):\ \begin{array}{c}s\in[i-j+1,i]_n,\\ j-d(i,s)\leq t\leq n-2-d(i,s)\end{array}\right\}\end{array}$$ where $d(a,b):= \#\{\text{arrows in the string } w=a\leftarrow \cdots \leftarrow b\}$ and $[i-j+1,i]_n$ is the \textbf{cyclic interval} in $[n]$ (i.e., there is a string $i \leftarrow (i-1) \leftarrow \cdots \leftarrow (i-j+2) \leftarrow (i-j+1)$ and the arithmetic is carried out mod $n$).
\end{lemma}

We now use Lemma~\ref{homstabhom} to classify extensions. By Lemma~\ref{dimext}, the dimension of $\Ext_{\Lambda}^1(X(k,\ell),X(i,j))$ is at most 1 for any indecomposables $X(k,\ell)$ and $X(i,j)$. Thus there is at most one nonsplit extension of the form $0\to X(i,j) \to Z \to X(k,\ell) \to 0$ up to equivalence of extensions.

\begin{proposition}\label{extclassif}
Let $i,k\in\Zbb/n\Zbb$ and $1\leq j,\ell<n-1$.  If $\Ext_{\Lambda}^1(X(k,\ell),X(i,j))\neq 0$, then

$\begin{array}{rl}
i) & \text{if } \supp(X(i,j)\cap X(k,\ell))=\emptyset, \text{ then the unique nonsplit extension is of the form }\\ & 0\ra X(i,j)\ra X(i,j+\ell)\ra X(k,\ell)\ra 0,\\
ii) & \text{if } \supp(X(i,j)\cap X(k,\ell))\neq\emptyset, \text{ then the unique nonsplit extension is of the form }\\ & 0\ra X(i,j)\ra X(i,d(i,k)+\ell)\oplus X(k,j-d(i,k))\ra X(k,\ell)\ra 0.
\end{array}$
\end{proposition}

\begin{proof}
By the Auslander-Reiten formula,
\begin{align*}
\dim\Ext_{\Lambda}^1(X(k,\ell),X(i,j)) &=\dim\underline{\Hom}_{\Lambda}(\tau^{-1}X(i,j),X(k,\ell))\\
&=\dim\underline{\Hom}_{\Lambda}(X(i-1,j),X(k,\ell)).
\end{align*}
Hence, if $\Ext_{\Lambda}^1(X(k,\ell),X(i,j))\neq 0$ then $\supp(X(i-1,j))\cap\supp(X(k,\ell))\neq\emptyset$.

i) Assume $\supp(X(i,j))\cap\supp(X(k,\ell))=\emptyset$, then $\supp(X(i-1,j))\cap\supp(X(k,\ell))=\{i-j\}$ since $\supp(X(i-1,j))\backslash\supp(X(i,j))=\{i-j\}$.  Since there is a nonzero morphism $X(i-1,j)\ra X(k,\ell)$, we must have $k=i-j$ by the description of morphisms in Lemma \ref{homstabhom} $a)$.  Then $X(k,\ell)=X(i-j,\ell)$ and the extension must be of the form
$$0\ra X(i,j)\ra X(i,j+\ell)\ra X(k,\ell)\ra 0.$$

ii) Assume $\supp(X(i,j))\cap\supp(X(k,\ell))\neq\emptyset$. Here it is enough to show that there are inclusions (resp., surjections) $X(i,j) \hookrightarrow X(i,d(i,k)+\ell)$ and $X(k,j-d(i,k)) \hookrightarrow X(k,\ell)$ (resp., $X(i,j) \twoheadrightarrow X(k,j-d(i,k))$ and $X(i,d(i,k)+\ell) \twoheadrightarrow X(k,\ell)$). 

To do so, we first show that $d(i,k) + \ell \le n-1$. Observe that $\ell \le n - 2 - d(i-1,k)$ by Lemma~\ref{homstabhom} $b)$ since $\underline{\Hom}_\Lambda(X(i-1,j),X(k,\ell)) \neq 0.$ Since $d(i,k) - d(i-1,k) = 1,$ we have that $d(i,k) + \ell \le n-1.$ Thus $X(i,d(i,k) + \ell) \in \text{ind}(\Lambda\text{-mod})$ and therefore there is an inclusion $X(i,j)\hookrightarrow X(i,d(i,k)+\ell).$ 

Next, we show that  and $j-d(i,k) \ge 1.$ By Lemma \ref{homstabhom} $a)$, we deduce that $k\in[i-j+1,i]_n$ since $k \neq i-j.$ Therefore, we conclude that $d(i,k) \le j-1$ so $j-d(i,k) \ge 1.$ Thus $X(k,j-d(i,k)) \in \text{ind}(\Lambda\text{-mod})$ so there is an inclusion $X(k,j-d(i,k))\hookrightarrow X(k,\ell)$.

Lastly, we show that the desired surjections exist. Observe that since $d(i,k) = \#\{\text{arrows in the string } i \leftarrow (i-1) \leftarrow \cdots \leftarrow (k+1) \leftarrow k\}$ we have that $i-d(i,k) = k$ where this equation holds mod $n$. Thus, by composing surjective irreducible morphisms, we obtain the desired surjections $$X(i,j) \twoheadrightarrow X(i-1,j-1) \twoheadrightarrow \cdots \twoheadrightarrow X(k+1,j+1-d(i,k)) \twoheadrightarrow X(k,j-d(i,k))$$ and $$X(i,d(i,k)+\ell) \twoheadrightarrow X(i-1,d(i,k)+\ell-1) \twoheadrightarrow \cdots \twoheadrightarrow X(k+1,\ell+1)\twoheadrightarrow X(k,\ell).$$ Hence, the unique extension is of the form
$$0\ra X(i,j)\ra X(i,d(i,k)+\ell)\oplus X(k,j-d(i,k))\ra X(k,\ell)\ra 0.$$\end{proof}

When we use this classification of extensions to prove our main results, we will want to use only the notation for string modules. Thus we give the following translation of Proposition~\ref{extclassif} using the notation for string modules.

\begin{lemma}\label{extq_n}
Let $Q = Q(n)$ for some $n \ge 3$ and let $\Lambda$ denote the corresponding cluster-tilted algebra. Let $M(w_2),M(w_1) \in \text{ind}(\Lambda\text{-mod})$ where $\Ext^1_\Lambda(M(w_1),M(w_2)) \neq 0$. Let $$\xi = 0 \to M(w_2) \to Z \to M(w_1) \to 0$$ denote the unique nonsplit extension up to equivalence of extensions with $Z \in \Lambda\text{-mod}$. Then either

$\begin{array}{rl}
i) & \text{supp}(M(w_2))\cap \text{supp}(M(w_1)) = \emptyset \text{ and } Z = M(w_2 \longleftarrow w_1) \text{ or}\\
ii) & w_2 = u \longleftarrow w \text{ with } w_1 = w \longleftarrow v \text{ for some strings } u, v, \text{ and } w \text{ where } w = x_1 \stackrel{\alpha_1}{\longleftarrow} x_2 \cdots x_{k-1} \stackrel{\alpha_{k-1}}{\longleftarrow} x_k \\
& \text{satisfies } \text{supp}(M(w_2)) \cap \text{supp}(M(w_1)) = \{x_i\}_{i \in [k]}, \ (Q(n))_0\backslash(\text{supp}(M(w_2)) \cap \text{supp}(M(w_1))) \neq \emptyset,  \text{ and }\\
& Z = M(u \leftarrow w \leftarrow v) \oplus M(w).
\end{array}$
\end{lemma}

\section{Lattice properties}\label{latticeprops}

In this section, we give some background on lattices.  After establishing notation in Section~\ref{subsec_lattice_basics}, we discuss semidistributive, congruence-uniform, and polygonal lattices in the remaining sections.

\subsection{Basic notions}\label{subsec_lattice_basics}

A \textbf{lattice} $L$ is a poset for which every pair of elements $x,y\in L$ has a least upper bound $x\vee y$ and greatest lower bound $x\wedge y$, called the \textbf{join} and \textbf{meet}, respectively.  A lattice is \textbf{complete} if meets and joins exist for arbitrary subsets of $L$.  We will mainly deal with finite lattices, where these two conditions coincide.  Any complete lattice has a \textbf{top element} $\bigvee L := \bigvee_{x \in L} x$ and \textbf{bottom element} $\bigwedge L := \bigwedge_{x \in L} x$, which we denote by $\hat{1}$ and $\hat{0}$, respectively.

Many properties of posets come in dual pairs.  Given a poset $(P,\leq)$, its \textbf{dual poset} $(P^{\op},\leq^{\op})$ has the same underlying set and $x\leq^{\op} y$ if and only if $y\leq x$.  If $P$ is a lattice, then $P^{\op}$ has the same lattice structure with $\wedge$ and $\vee$ swapped.

A \textbf{lattice congruence} $\Theta$ on a lattice $L$ is an equivalence relation that respects the lattice operations; i.e., for $x,y,z\in L$, $x\equiv y\mod{\Theta}$ implies $(x\vee z)\equiv(y\vee z)\mod{\Theta}$ and $(x\wedge z)\equiv(y\wedge z)\mod{\Theta}$.  The lattice operations on $L$ induce a lattice structure on the set of equivalence classes of $\Theta$, which we denote $L/\Theta$.  The natural map $L\ra L/\Theta$ is a \textbf{lattice quotient map}.

Figure \ref{fig_lattice_quotient} contains two examples of lattice quotient maps.  The blue arrows in each of the upper lattices are contracted to form the lower lattices.

To prove that a given equivalence relation is a lattice congruence, we will make use of the following well-known result.

\begin{lemma}\cite[Exercise 9.37]{ReadingPAB}\label{lattconglemma}
Let $L$ be a finite lattice with idempotent, order-preserving maps $\pi_{\downarrow},\pi^{\uparrow}:L\ra L$ that satisfy $\pi_\downarrow(x) \le x \le \pi^\uparrow(x)$ for any $x \in L$.  Let $\Theta$ be the equivalence relation $x\equiv y\mod{\Theta}$ if $\pi_{\downarrow}(x)=\pi_{\downarrow}(y)$.  If $\pi_{\downarrow}\circ\pi^{\uparrow}=\pi_{\downarrow}$ and $\pi^{\uparrow}\circ\pi_{\downarrow}=\pi^{\uparrow}$, then $\Theta$ is a lattice congruence of $L$.
\end{lemma}

The maps $\pi_{\downarrow}$ and $\pi^\uparrow$ are idempotent endomorphisms of $L$.  However, we may identify $\pi_{\downarrow}$ with the natural lattice quotient map $L\ra L/\Theta$ when convenient.

\begin{figure}
\includegraphics{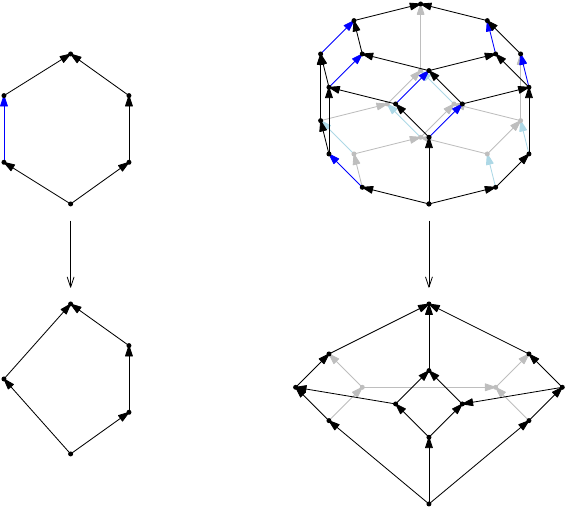}
\caption{\label{fig_lattice_quotient}Two examples of lattice quotient maps.}
\end{figure}

An element $j$ of a lattice $L$ is \textbf{join-irreducible} (dually, \textbf{meet-irreducible}) if $j\neq\hat{0}$ and for $x,y\in L$, $j=x\vee y$ implies $j=x$ or $j=y$.  For finite lattices, an element $j$ is join-irreducible exactly when it covers a unique element, denoted $j_*$.  Dually, a meet-irreducible element $m$ is covered by a unique element $m^*$.  We let $J(L)$ (resp., $M(L)$) be the set of join-irreducible (resp., meet-irreducible) elements of $L$.

\subsection{Semidistributive lattices}

A lattice $L$ is \textbf{meet-semidistributive} if for $x,y,z\in L$,
$$(x\vee y)\wedge z=x\wedge z\hspace{5mm}\mbox{ holds whenever }\hspace{5mm}x\wedge z=y\wedge z.$$
A lattice is \textbf{join-semidistributive} if its dual is meet-semidistributive.  It is \textbf{semidistributive} if it is both join-semidistributive and meet-semidistributive.   Clearly, every distributive lattice is semidistributive.  On the other hand, the five-element lattice of Figure \ref{fig_lattice_quotient} is semidistributive but not distributive.  As semidistributivity is defined by equations in the lattice operations, it is preserved under lattice quotients.

For a finite lattice $L$, any element $x$ admits a representation of the form $x=\bigvee A$ where $A$ is a subset of $J(L)$.  The representation is \textbf{irredundant} if $x>\bigvee A^{\pr}$ for any proper subset $A^{\pr}$ of $A$.  Given $A,B\subseteq J(L)$, we say $A\leq B$ if every element of $A$ is less than or equal to some element of $B$.  A join-representation $x=\bigvee A$ for $x\in L,\ A\subseteq J(L)$ is called a \textbf{canonical join-representation} if it is irredundant and $A\leq B$ whenever $B\subseteq J(L)$ with $x=\bigvee B$.  \textbf{Canonical meet-representations} are defined dually. The following lemma gives an explicit definition of canonical meet-representations. 

\begin{lemma}\label{canonjoincanonmeet}
Given an element $x$ of a lattice $L$, the expression $x=\bigvee_{i=1}^l j_i$ is a canonical join-representation of $x$ in $L$ if and only if $x^{\op}=\bigwedge_{i=1}^lj_i^{\op}$ is a canonical meet-representation in the dual lattice $L^{\op}$.
\end{lemma}

A finite semidistributive lattice admits canonical join-representations for all of its elements \cite[Theorem 2.24]{freese1995free}.  These canonical join-representations often take a very nice form.  One of our main applications of semidistributivity is a description of canonical join-representations and canonical meet-representations of torsion classes (see Theorem~\ref{canonjoinrepn} and Corollary~\ref{canonmeetrepn}).

\subsection{Congruence-uniform lattices}

Let $P$ and $\{0,1\}$ be posets where $\{0,1\}$ is a two element chain with $0 < 1$, and consider the poset $P \times \{0,1\}$ with partial order $(x,i) \le (y,j)$ if $x \le y$ in $P$ and $i \le j$ in $\{0,1\}$.   Now given a closed interval $I=[x,y]$ of $P$, the \textbf{doubling} $P[I]$ is the induced subposet of $P\times\{0,1\}$ with elements
$$P[I]=(P_{\leq y}\times\{0\})\sqcup[(P-P_{\leq y})\cup I]\times\{1\},$$
where $P_{\leq y}=\{z\in P:\ z\leq y\}$.  As shown in \cite{day1992doubling}, if $P$ is a lattice, then $P[I]$ is a lattice.  A finite lattice $L$ is \textbf{congruence-uniform} (or \textbf{bounded}) if there exists a sequence of lattices $L_1,\ldots,L_t$ such that $L_1$ is the one-element lattice, $L_t=L$, and for all $i$, there exists a closed interval $I_i$ of $L_i$ such that $L_{i+1}\cong L_i[I_i]$.

As interval doublings preserve semidistributivity, finite congruence-uniform lattices are always semidistributive.  Congruence-uniform lattices admit other characterizations in terms of lattice congruences \cite{day:congruence}, edge-labelings \cite{reading:lattice}, or as ``bounded'' quotients of free lattices.

\subsection{Polygonal lattices}

A finite lattice is a \textbf{polygon} if it contains exactly two maximal chains and those chains only agree at the bottom and top elements.  A finite lattice $L$ is \textbf{polygonal} if for all $x\in L$:
\begin{itemize}
\item if $y$ and $z$ are distinct elements covering $x$, then $[x,y\vee z]$ is a polygon, and
\item if $y$ and $z$ are distinct elements covered by $x$, then $[y\wedge z,x]$ is a polygon.
\end{itemize}

Given two maximal chains $C,C^{\pr}$ in a lattice $L$, we say $C$ and $C^{\pr}$ differ by a \textbf{polygonal flip} if there is a polygon $[x,y]$ such that $C\cap[\hat{0},x]=C^{\pr}\cap[\hat{0},x]$, $C\cap[y,\hat{1}]=C^{\pr}\cap[y,\hat{1}]$ and $C\cap[x,y]$ and $C^{\pr}\cap[x,y]$ are distinct maximal chains of $[x,y]$.

Our main use of polygonal lattices is the following connectivity result.

\begin{lemma}\cite[Lemma 9-6.3]{ReadingPAB}\label{lem_poly_flip}
Let $L$ be a polygonal lattice.  If $C$ and $C^{\pr}$ are maximal chains of $L$, then there exists a sequence of maximal chains $C=C_0,C_1,\ldots,C_N=C^{\pr}$ such that $C_i$ and $C_{i+1}$ differ by a polygonal flip for all $i$.
\end{lemma}

\section{Semidistributivity of oriented exchange graphs}\label{semidisoreg}

In this section, we prove that if $Q$ is mutation-equivalent to a Dynkin quiver, then $\overrightarrow{EG}(\widehat{Q})$ is a semidistributive lattice. To do so, we begin by identifying $\overrightarrow{EG}(\widehat{Q})$ with the lattice of functorially finite torsion classes of the cluster-tilted algebra $\Lambda = \Bbbk Q/I$ (see \cite{by14}). After that, we prove that the lattice of torsion classes of any finite dimensional algebra is semidistributive. Since $\Lambda$ is representation-finite, all torsion classes are functorially finite and thus we conclude that $\overrightarrow{EG}(\widehat{Q})$ is semidistributive.

\subsection{Torsion classes and oriented exchange graphs} In this brief section, we recall the definition of torsion classes and the connection between torsion classes and oriented exchange graphs. 

Let $\Lambda$ be a finite dimensional $\Bbbk$-algebra. A full, additive subcategory $\mathcal{C} \subset \Lambda$-mod is \textbf{extension closed} if for any objects $X,Y \in \mathcal{C}$ and $0 \to X \to Z \to Y \to 0$ exact one has $Z \in \mathcal{C}$. We say $\mathcal{C}$ is \textbf{quotient closed} (resp., \textbf{submodule closed}) if for any $X \in \mathcal{C}$ satisfying $X \stackrel{\alpha}{\longrightarrow} Z$ where $\alpha$ is a surjection (resp., $Z \stackrel{\beta}{\longrightarrow} X$ where $\beta$ is an injection), then $Z \in \mathcal{C}$. A full, additive subcategory $\mathcal{T} \subset \Lambda$-mod is called a \textbf{torsion class} if $\mathcal{T}$ is quotient closed and extension closed. Dually, a full, additive subcategory $\mathcal{F} \subset \Lambda$-mod is called a \textbf{torsion-free class} if $\mathcal{F}$ is extension closed and submodule closed. 

Let $\text{tors}(\Lambda)$ (resp., $\text{torsf}(\Lambda)$) denote the poset of torsion classes  (resp., of torsion-free classes) of $\Lambda$ ordered by inclusion. We have the following proposition, which shows that a torsion class of $\Lambda$ uniquely determines a torsion-free class of $\Lambda$ and vice versa.

\begin{proposition}\label{torsbij}
\cite[Proposition 2.1 a)]{iyama.reiten.thomas.todorov:latticestrtors} The maps $$\begin{array}{rcl} \text{tors}(\Lambda) & \stackrel{(-)^\perp}{\longrightarrow} & \text{torsf}(\Lambda)\\ \mathcal{T} & \longmapsto &  \mathcal{T}^\perp := \{X \in \Lambda\text{-mod}: \ \Hom_{\Lambda}(\mathcal{T},X) = 0\}\end{array}$$
and
$$\begin{array}{rcl} \text{torsf}(\Lambda) & \stackrel{^{\perp}(-)}{\longrightarrow} & \text{tors}(\Lambda)\\ \mathcal{F} & \longmapsto &  {}^{\perp}\mathcal{F} := \{X \in \Lambda\text{-mod}: \ \Hom_{\Lambda}(X,\mathcal{F}) = 0\}\end{array}$$ are inverse bijections. 
\end{proposition} 

We also have the following useful lemma.

\begin{lemma}\cite[Proposition 2.4]{iyama.reiten.thomas.todorov:latticestrtors}\label{standarddual}
The maps $$\begin{array}{rcl} \text{tors}(\Lambda) & \stackrel{D(-)}{\longrightarrow} & \text{torsf}(\Lambda^{\text{op}})\\ \mathcal{T} & \longmapsto &  D{\mathcal{T}} \end{array}$$
and
$$\begin{array}{rcl} \text{torsf}(\Lambda) & \stackrel{D(-)}{\longrightarrow} & \text{tors}(\Lambda^{\text{op}})\\ \mathcal{F} & \longmapsto &  D\mathcal{F} \end{array}$$ are isomorphisms of lattices where $D(-):= \Hom_\Bbbk(-,\Bbbk)$ is the \textbf{standard duality}. Furthermore, the functor $D((-)^\perp): \text{tors}(\Lambda) \to \text{tors}(\Lambda^{\op})$ is an anti-isomorphism of posets.
\end{lemma}

The lattices $\text{tors}(\Lambda)$ and $\text{torsf}(\Lambda)$ have the following description of their meet and join operations. In Lemma~\ref{join}, we give an alternative description of the join operation. 

\begin{proposition}\cite[Proposition 2.3]{iyama.reiten.thomas.todorov:latticestrtors}\label{meetandjointors} Let $\Lambda$ be a finite dimensional algebra. Then $\text{tors}(\Lambda)$ and $\text{torsf}(\Lambda)$ are complete lattices. The join and meet operations are as follows.
\begin{itemize}
\item[a)] Let $\{\mathcal{T}_i\}_{i \in I} \subset \text{tors}(\Lambda)$ be a collection of torsion classes. Then we have $\bigwedge_{i\in I}\mathcal{T}_i = \bigcap_{i \in I} \mathcal{T}_i$ and $\bigvee_{i\in I}\mathcal{T}_i = {}^{\perp}\left(\bigcap_{i \in I} \mathcal{T}_i^\perp\right)$.
\item[b)] Let $\{\mathcal{F}_i\}_{i \in I} \subset \text{torsf}(\Lambda)$ be a collection of torsion-free classes. Then we have $\bigwedge_{i\in I}\mathcal{F}_i = \bigcap_{i \in I} \mathcal{F}_i$ and $\bigvee_{i\in I}\mathcal{F}_i = \left(\bigcap_{i \in I} {}^{\perp}\mathcal{F}_i\right)^{\perp}$.
\end{itemize}

\end{proposition}

An important subset of $\text{tors}(\Lambda)$ is the set of functorially finite torsion classes, denoted $\text{f-tors}(\Lambda)$. By definition, $\mathcal{T} \in \text{tors}(\Lambda)$ is a \textbf{functorially finite} torsion class if there exists $X \in \Lambda$-mod such that $$\mathcal{T} = \text{Fac}(X) := \{Y \in \Lambda\text{-mod}: \exists X^m \twoheadrightarrow Y \text{ for some } m \in \mathbb{N}\}.$$ Dually, a torsion-free class $\mathcal{F} \in \text{tors}(\Lambda)$ is \textbf{functorially finite} if there exists $X \in \Lambda$-mod such that $$\mathcal{F} = \text{Sub}(X) := \{Y \in \Lambda\text{-mod}: \exists Y \hookrightarrow X^m \text{ for some } m \in \mathbb{N}\}.$$ We let $\text{f-torsf}(\Lambda)$ denote the set of functorially finite torsion-free classes of $\Lambda.$  The sets $\text{f-tors}(\Lambda)$ and $\text{f-torsf}(\Lambda)$ are clearly partially-ordered by inclusion. The bijection given in Proposition~\ref{torsbij} restricts to a bijection $\text{f-tors}(\Lambda) \to \text{f-torsf}(\Lambda).$

Now, let $\Lambda = \Bbbk Q/I$ be a cluster-tilted algebra. The next result, which appears in \cite{by14} in much greater generality, shows that oriented exchange graphs can be studied using functorially finite torsion classes of $\Lambda$.

\begin{proposition}\label{oregequalsftors}
Let $Q$ be a quiver that is mutation-equivalent to a Dynkin quiver and let $\Lambda = \Bbbk Q/I$ denote the associated cluster-tilted algebra. Then the Hasse diagram of $\overrightarrow{EG}(\widehat{Q})$ is isomorphic to that of $\text{f-tors}(\Lambda)$. Using this isomorphism, we may regard $\overrightarrow{EG}(\widehat{Q})$ as a poset.
\end{proposition}

\begin{theorem}\label{semidis}
If $\Lambda$ is any finite dimensional $\Bbbk$-algebra, then the lattice $\text{tors}(\Lambda)$ is semidistributive. In particular, if $\Lambda$ is a finite dimensional $\Bbbk$-algebra of finite representation type, then $\text{f-tors}(\Lambda)$ is a semidistributive lattice.
\end{theorem}

\begin{proof}
It is enough to show that $\text{tors}(\Lambda)$  is \textbf{meet-semidistributive} (i.e., for any $\mathcal{T}_1, \mathcal{T}_2, \mathcal{T}_3 \in \text{tors}(\Lambda)$ satisfying $\mathcal{T}_1 \wedge \mathcal{T}_3 = \mathcal{T}_2 \wedge \mathcal{T}_3$ we have that $(\mathcal{T}_1 \vee \mathcal{T}_2)\wedge \mathcal{T}_3 = \mathcal{T}_1 \wedge \mathcal{T}_3$) since $\text{tors}(\Lambda)$ is {join-semidistributive} if and only if $\text{tors}(\Lambda^\text{op})$ is meet-semidistributive. This is proved in the next section (see Lemma~\ref{meetsemi}).

It is well-known that $\text{f-tors}(\Lambda) = \text{tors}(\Lambda)$ holds when $\Lambda$ is a finite dimensional, representation-finite $\Bbbk$-algebra. Thus the second assertion holds. 
\end{proof}

\begin{remark}
In \cite[Theorem 1.2]{dij}, finite dimensional algebras $\Lambda$ satisfying $\text{tors}(\Lambda) = \text{f-tors}(\Lambda)$ are shown to be exactly those algebras that are $\tau$-\textbf{tilting finite} algebras. That is, algebras $\Lambda$ with only finitely many $\tau$\textbf{-rigid modules} (i.e., modules $M$ satisfying $\Hom_\Lambda(M, \tau M) = 0$). Additionally, in \cite[Theorem 1.2]{iyama.reiten.thomas.todorov:latticestrtors} it is shown that $\text{f-tors}(\Lambda)$ is a complete lattice if and only if $\Lambda$ is a $\tau$-tilting finite algebra.
\end{remark}

\begin{corollary}
Let $Q$ be a quiver that is mutation-equivalent to a Dynkin quiver. Then $\overrightarrow{EG}(\widehat{Q})$ is a semidistributive lattice.
\end{corollary}
\begin{proof}
Since $Q$ defines a representation-finite cluster-tilted algebra $\Lambda$, we know that $\text{f-tors}(\Lambda) = \text{tors}(\Lambda).$ By Proposition~\ref{oregequalsftors}, we have that $\overrightarrow{EG}(\widehat{Q}) \cong \text{tors}(\Lambda).$ By Theorem~\ref{semidis}, we now have that $\overrightarrow{EG}(\widehat{Q})$ is semidistributive.
\end{proof}

\begin{example}\label{torsexample}
Let $Q$ be the quiver appearing in Example~\ref{stringsA3}. Note that $Q = Q(3)$. We show the Auslander-Reiten quiver of the cluster-tilted algebra $\Lambda = \Bbbk Q(3)/\langle \alpha_1\alpha_2: \ \alpha_i \in Q(3)_1\rangle$ below (see Figure~\ref{arQ3}) and use it to describe the oriented exchange graph of $Q$ as the lattice of torsion classes and torsion-free classes of $\Lambda$. Any $\mathcal{T} \in \text{tors}(\Lambda)$ or $\mathcal{F} \in \text{torsf}(\Lambda)$ is \textbf{additively generated} (i.e., a full, additive subcategory $\mathcal{A}$ of $\Lambda$-mod is \textbf{additively generated} if $\mathcal{A} =\text{add}(\oplus_{i \in [k]}M(w_i))$ for some finite subset $\{M(w_i)\}_{i \in [k]} \subset \text{ind}(\Lambda\text{-mod})$ where $\text{add}(\oplus_{i \in [k]}M(w_i))$ is the smallest full, additive subcategory closed under direct summands of $\Lambda$-mod that contains $\{M(w_i)\}_{i \in [k]}$) so $\mathcal{T}$ and $\mathcal{F}$ are completely determined by the set of indecomposable modules they contain. Using this fact, we show torsion classes of $\Lambda$ (resp., torsion-free classes of $\Lambda$) in blue (resp., red). For example, $\mathcal{T} = \text{add}(X(3,2)\oplus X(2,1))$ and its corresponding torsion-free class $\mathcal{F} = \text{add}(X(1,1)\oplus X(1,2) \oplus X(3,1))$ are depicted in Figure~\ref{add}. We show all of the torsion classes and torsion-free classes of $\Lambda$ in Figure~\ref{A3tors}.

\begin{figure}[h]
$$\begin{array}{rcl}
\raisebox{-.5in}{$\Gamma(\Lambda\text{-mod})$}
& \raisebox{-.5in}{$=$} & \raisebox{.1in}{$\begin{xy} 0;<1pt,0pt>:<0pt,-1pt>:: 
(0,60) *+{X(3,1)} ="0",
(60,60) *+{X(2,1)} ="1",
(120,60) *+{X(1,1)} ="2",
(180,60) *+{X(3,1)} ="3",
(30,30) *+{X(3,2)} ="5",
(90,30) *+{X(2,2)} ="8",
(150,30) *+{X(1,2)} ="10",
"1", {\ar^{\tau}@{-->}"0"},
"0", {\ar@{^{(}->}"5"},
"2", {\ar^\tau@{-->}"1"},
"5", {\ar@{->>}"1"},
"1", {\ar@{^{(}->}"8"},
"3", {\ar^\tau@{-->}"2"},
"8", {\ar@{->>}"2"},
"2", {\ar@{^{(}->}"10"},
"10", {\ar@{->>}"3"},
\end{xy}$}\end{array}$$
\caption{The Auslander-Reiten quiver of $\Lambda$.}\label{arQ3}
\end{figure}

\begin{figure}[h]$$\includegraphics[scale=1.25]{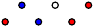}$$\caption{$\mathcal{T} = \text{add}(X(3,2)\oplus X(2,1))$ and $\mathcal{F} = \text{add}(X(1,1)\oplus X(1,2) \oplus X(3,1))$}\label{add}\end{figure}  

\begin{figure}[h]
$$\begin{array}{cccc}
\includegraphics[scale=1.3]{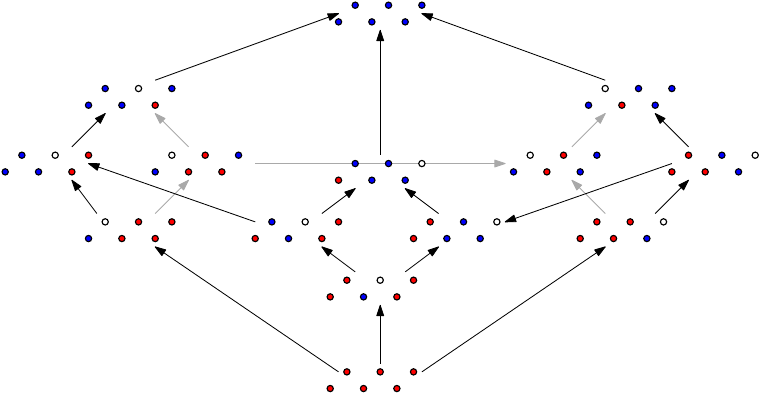}
\end{array}$$
\caption{The oriented exchange graph of $Q(3)$ modeled using $\text{tors}(\Lambda)$ and $\text{torsf}(\Lambda)$.}
\label{A3tors}
\end{figure}

\end{example}

\subsection{Meet-semidistributivity of $\text{tors}(\Lambda)$}\label{Sec_42} In this section, we prove that the lattice of torsion classes of a finite dimensional $\Bbbk$-algebra $\Lambda$ is meet-semidistributive. As a preliminary step, we give an explicit description of the join of two torsion classes (see Lemma~\ref{join}). We thank Hugh Thomas for mentioning this description of the join to us \cite{hthomas}.

\begin{lemma}\label{join}
If $\mathcal{T}, \mathcal{U} \in \text{tors}(\Lambda)$. Then $$\mathcal{T}\vee \mathcal{U} = \mathcal{F}ilt(\mathcal{T}\cup\mathcal{U})$$ 
where $\mathcal{F}ilt(\mathcal{T}\cup\mathcal{U})$ is defined as the subcategory of all $\Lambda$-modules $X$ with a filtration $0 = X_0 \subset X_1 \subset \cdots \subset X_n = X$ with the property that $X_j/X_{j-1}$ belongs to $\mathcal{T}$ or $\mathcal{U}$ for any $j \in [n]$.
\end{lemma}

Lemma~\ref{join} follows from the general fact that if $\mathcal{C}$ is a class of $\Lambda$-modules which is quotient closed, then the category $\mathcal{F}ilt(\mathcal{C})$ belongs to $\text{tors}(\Lambda)$, which was observed in \cite[Proposition 3.3]{dij}. We are grateful to an anonymous referee for bringing this to our attention. We now complete the proof that $\text{tors}(\Lambda)$ is meet-semidistributive using Lemma~\ref{join}.

\begin{lemma}\label{meetsemi}
The lattice tors($\Lambda$) is meet-semidistributive.
\end{lemma}
\begin{proof}
Let $\mathcal{T}_1, \mathcal{T}_2, \mathcal{T}_3 \in \text{tors}(\Lambda).$ We show that if $\mathcal{T}_1\wedge \mathcal{T}_3 = \mathcal{T}_2 \wedge \mathcal{T}_3$, then $(\mathcal{T}_1 \vee \mathcal{T}_2)\wedge \mathcal{T}_3 = \mathcal{T}_1 \wedge \mathcal{T}_3.$ It is clear that $(\mathcal{T}_1 \vee \mathcal{T}_2)\wedge \mathcal{T}_3 \supset \mathcal{T}_1 \wedge \mathcal{T}_3$ so it is enough to show $(\mathcal{T}_1 \vee \mathcal{T}_2)\wedge \mathcal{T}_3 \subset \mathcal{T}_1 \wedge \mathcal{T}_3.$

Let $X \in (\mathcal{T}_1 \vee \mathcal{T}_2)\wedge \mathcal{T}_3$. By Lemma~\ref{join}, we see that $X$ has a filtration $0 = X_0 \subset X_1 \subset \cdots \subset X_n = X$ where for each $i \in [n]$ the quotient $X_i/X_{i-1}$ belongs to $\mathcal{T}_1$ or $\mathcal{T}_2.$ 

To complete the proof, it is enough to show that the module $X/X_{n-j} \in \mathcal{T}_1 \wedge \mathcal{T}_2 \wedge \mathcal{T}_3$ for any $j \in [n]$. Since $\mathcal{T}_3$ is a torsion class and since there is a surjection $X \to X/X_{n-j}$ for any $j \in [n]$, it is clear that $X/X_{n-j} \in \mathcal{T}_3$ for any $j \in [n].$ Thus we need to show that $X/X_{n-j} \in \mathcal{T}_1 \wedge \mathcal{T}_2$ for any $j \in [n]$.

To show that $X/X_{n-j} \in \mathcal{T}_1\wedge\mathcal{T}_2$ for any $j \in \{0\}\cup[n]$, we proceed by induction on $j$. The statement is clear when $j=0$. Assume $j\geq 1$ and $X/X_{n-j+1} \in \mathcal{T}_1 \wedge \mathcal{T}_2$ holds. We have a short exact sequence $$0 \to X_{n-j+1}/X_{n-j} \to X/X_{n-j} \to X/X_{n-j+1} \to 0$$ where $X_{n-j+1}/X_{n-j}$ belongs to $\mathcal{T}_1$ or $\mathcal{T}_2$ by the properties of the filtration of $X$. Thus $X/X_{n-j} \in \mathcal{T}_1$ or $\mathcal{T}_2$ since torsion classes are extension closed. By assumption, $\mathcal{T}_1\wedge \mathcal{T}_3 = \mathcal{T}_2 \wedge \mathcal{T}_3$ so $X/X_{n-j} \in \mathcal{T}_1\wedge \mathcal{T}_2.$ We conclude that $X/X_{n-j}\in \mathcal{T}_1\wedge \mathcal{T}_2$ for any $j \in [n].$\end{proof}

\section{Biclosed sets}\label{biclosedsets}

A \textbf{closure operator} on a set $C$ is an operator $X\mapsto\ov{X}$ on subsets of $C$ such that for $X,Y\subseteq C$:
\begin{itemize}
\item $X\subseteq\ov{X}$,
\item $\ov{\ov{X}}=\ov{X}$, and
\item if $X\subseteq Y$, then $\ov{X}\subseteq\ov{Y}$.
\end{itemize}
In addition, we assume that $\ov{\emptyset}=\emptyset$.  A subset $X$ of $C$ is \textbf{closed} if $X=\ov{X}$.  It is \textbf{co-closed} (or \textbf{open}) if $C-X$ is closed.  We say $X$ is \textbf{biclosed} (or \textbf{clopen}) if it is both closed and co-closed.  We let $\Bic(C)$ denote the poset of biclosed subsets of $C$, ordered by inclusion. If $C^{\pr}\subseteq C$, then $C^{\pr}$ inherits a closure operator from $C$; namely, $X\subseteq C^{\pr}$ is \textbf{closed relative to} $C^{\pr}$ if $X=\ov{X}\cap C^{\pr}$. One may define being relatively co-closed in a similar manner. In this way, we may consider $\Bic(C^{\pr})$ as the collection of subsets of $C^{\pr}$ that are closed and co-closed relative to $C^{\pr}$. We remark that the map $\Bic(C)\ra\Bic(C^{\pr})$ where $X\mapsto X\cap C^{\pr}$ is well-defined, but it may not be surjective in general.

For many closure operators, the poset of biclosed sets is not a lattice.  However, in some special cases, $\Bic(C)$ is a lattice with a semidistributive or congruence-uniform structure.  For example, if $C$ is the set of positive roots of a finite root system endowed with the convex closure, then $\Bic(C)$ is a congruence-uniform lattice \cite{reading:lattice}.  In this setting, biclosed sets of positive roots are inversion sets of elements of the associated Coxeter group, so $\Bic(C)$ may be identified with the weak order.

Some sufficient (but not necessary) criteria for semidistributivity and congruence-uniformity were given in \cite{mcconville:grassmann}. The statement about polygonality is new, so we provide a proof. We say a collection $\Bcal$ of subsets of $C$ is \textbf{ordered by single-step inclusion} if whenever $X,Y\in\Bcal$ with $X\subsetneq Y$, there exists $c\in Y-X$ such that $X\cup\{c\}\in\Bcal$.

\begin{theorem}\cite[Theorem 5.2]{mcconville:grassmann}\label{thm_cu_criteria}
Let $C$ be a set with a closure operator.  Assume that
\begin{enumerate}
\item\label{thm_cu_criteria_1} $\Bic(C)$ is ordered by single-step inclusion, and
\item\label{thm_cu_criteria_2} $W\cup\ov{(X\cup Y)-W}$ is biclosed for $W,X,Y\in\Bic(C)$ with $W\subseteq X\cap Y$.

Then $\Bic(C)$ is a semidistributive lattice.

$\Bic(C)$ is congruence-uniform if it satisfies (1), (2), and there is a poset structure $(C,\prec)$ such that

\item\label{thm_cu_criteria_3} if $x,y,z\in C$ with $z\in\ov{\{x,y\}}-\{x,y\}$ then $x\prec z$ and $y\prec z$.

$\Bic(C)$ is polygonal if it satisfies (1), (2), and 

\item\label{thm_cu_criteria_4} for distinct $x,y\in C$, $\Bic(\ov{\{x,y\}})$ is a polygon.

\end{enumerate}

\end{theorem}

\begin{proof}
  Only the statement about polygonality remains to be proved. Let $W,X,Y\in\Bic(C)$ be distinct biclosed sets such that $W$ is covered by $X$ and $Y$. By (\ref{thm_cu_criteria_1}), this means there exists $x,y\in C$ such that $X=W\cup\{x\}$ and $Y=W\cup\{y\}$. By (\ref{thm_cu_criteria_2}), the set $W\cup\ov{\{x,y\}}$ is biclosed and is the smallest closed set containing both $X$ and $Y$. Hence, $X\vee Y=W\cup\ov{\{x,y\}}$.

  The map $[W,X\vee Y]\ra\Bic(\ov{\{x,y\}})$ where $Z\mapsto Z- W$ is injective. We prove that it is surjective as well. This statement combined with (\ref{thm_cu_criteria_4}) then implies that $[W,X\vee Y]$ is a polygon.

  There exist chains $X_0\lessdot X_1\lessdot\cdots\lessdot X_l$ and $Y_0\lessdot Y_1\lessdot\cdots\lessdot Y_l$ in $[W,X\vee Y]$ such that $X_1=X,\ Y_1=Y,\ X_l=Y_l=X\vee Y$, and $X_0=Y_0=W$. Applying the map $Z\mapsto Z- W$ to both chains gives two chains in $\Bic(\ov{\{x,y\}})$, one containing $\{x\}$ and the other containing $\{y\}$. By (\ref{thm_cu_criteria_1}), these chains are unrefinable. As $\Bic(\ov{\{x,y\}})$ is a polygon, every element of $\Bic(\ov{\{x,y\}})$ must be of the form $X_i- W$ or $Y_i- W$ for some $i$, as desired.

  The poset $\Bic(C)$ is self-dual in the sense that $X\leq Y$ if and only if $C- Y\leq C- X$. From this duality, it follows from the preceding argument that $[X\wedge Y,W]$ is a polygon whenever $X$ and $Y$ are distinct biclosed sets both covered by $W$. Hence, $\Bic(C)$ is a polygonal lattice.
\end{proof}

\begin{example}\label{ex_perm}
For $X\subseteq\binom{[n]}{2}$, say $X$ is closed if $\{i,k\}\in X$ holds whenever $\{i,j\}\in X$ and $\{j,k\}\in X$ for $1\leq i<j<k\leq n$.  It is easy to check that biclosed subsets of $\binom{[n]}{2}$ are inversion sets of permutations.  Moreover, ordering $\{j,k\}\preceq\{i,l\}$ if $i\leq j<k\leq l$, this closure space satisfies the hypotheses of Theorem \ref{thm_cu_criteria}.  Hence, one may deduce that the weak order is a congruence-uniform and polygonal lattice.  We refer to \cite{mcconville:grassmann} for more examples.
\end{example}

\subsection{Biclosed sets of paths}

Given an undirected graph $G=(V,E)$, we say a set $p\subseteq V$ is an \textbf{acyclic path} if the induced subgraph of $G$ on $p$ is a path graph; i.e. it is a full subgraph of $G$ of type $\mathbb{A}_l$ for some $l>0$. Let $\AP$ denote the set of all acyclic paths of $G$. Two acyclic paths $p,p^{\pr}$ are said to be \textbf{composable} if $p\cap p^{\pr}=\emptyset$ and $p\cup p^{\pr}$ is in $\AP$. If $p$ and $p^{\pr}$ are composable, we set $p\circ p^{\pr}=p\cup p^{\pr}$. It is often convenient to consider an acyclic path as a sequence of distinct vertices $(v_0,\ldots,v_l)$ where $v_i$ and $v_j$ are adjacent if and only if $|i-j|=1$. In this case, composition of two acyclic paths $(v_0,\ldots,v_t)$ and $(v_0^{\pr},\ldots,v_s^{\pr})$ may be expressed as a sequence $(v_0,\ldots,v_t,v_0^{\pr},\ldots,v_s^{\pr})$. This notation can be misleading, however, since we consider $(v_0,\ldots,v_t)$ and its reverse $(v_t,\ldots,v_0)$ as the same acyclic path.

For $X\subseteq\AP$, we say $X$ is \textbf{closed} if for $p,p^{\pr}\in X$, if $p\circ p^{\pr}\in\AP$ then $p\circ p^{\pr}\in X$.  As before, we say $X$ is biclosed if both $X$ and $\AP-X$ are closed.

The closure of any subset of acyclic paths may be computed by successively concatenating paths.  We record this useful fact in the following lemma.

\begin{lemma}
Let $X\subseteq\AP$.  If $p\in\ov{X}$, then there exist paths $q_1,\ldots,q_t\in X$ such that $p=q_1\circ\cdots\circ q_t$.
\end{lemma}

\begin{theorem}\label{bicapprops}
$\Bic(\AP)$ is a semidistributive, congruence-uniform, and polygonal lattice.
\end{theorem}

\begin{proof}
To prove this result, we verify properties (1)-(4) of Theorem \ref{thm_cu_criteria}.

Let $X,Y\in\Bic(\AP)$ such that $X\subsetneq Y$.  If $p,q,q^{\pr}$ are paths such that $p\in Y$ and $q\circ q^{\pr}=p$, then either $q$ or $q^{\pr}$ is in $Y$.  If $p\in Y-X$ is chosen of minimum length, then if $p = q \circ q^{\pr}$, either $q$ or $q^{\pr}$ must be in $X$.

Among the elements $p$ of $Y-X$ such that if $p=q\circ q^{\pr}$ then either $q\in X$ or $q^{\pr}\in X$, choose $p_0$ to be of maximum length.  We prove that $X\cup\{p_0\}$ is biclosed.  By the choice of $p_0$, it is immediate that $X\cup\{p_0\}$ is co-closed.

Assume that $X\cup\{p_0\}$ is not closed.  Then there exists $p\in X$ such that $p\circ p_0$ is an acyclic path but is not in $X$.  Among such paths, we assume $p$ is of minimum length.  Since $Y$ is closed, $p\circ p_0$ is in $Y$.  By the maximality of $p_0$, there exist acyclic paths $q,q^{\pr}$ both not in $X$ such that $p\circ p_0=q\circ q^{\pr}$.  Let $p\circ p_0=(v_0,\ldots,v_t)$.  Up to path reversal, we may assume $p=(v_0,\ldots,v_{i-1}),\ p_0=(v_i,\ldots,v_t),\ q=(v_0,\ldots,v_{j-1}),\ q^{\pr}=(v_j,\ldots,v_t)$ for some distinct indices $i,j$.

If $i<j$, then $p\in X,\ q\notin X$ implies $(v_i,\ldots,v_{j-1})\notin X$ since $X$ is closed.  But, $(v_i,\ldots,v_{j-1})\circ q^{\pr}=p_0$, contradicting the choice of $p_0$.

If $j<i$, then $p\in X,\ q\notin X$ implies $(v_j,\ldots,v_{i-1})\in X$ since $X$ is co-closed.  But $(v_j,\ldots,v_{i-1})\circ p_0=q^{\pr}$, contradicting the minimality of $p$.

We conclude that $X\cup\{p_0\}$ is biclosed.  Hence, $\Bic(\AP)$ is ordered by single-step inclusion.  This completes the proof of (\ref{thm_cu_criteria_1}).

Now let $W,X,Y\in\Bic(\AP)$ such that $W\subseteq X\cap Y$.

Assume $W\cup\ov{(X\cup Y)-W}$ is not closed.  Choose $p,q\in W\cup\ov{(X\cup Y)-W}$ such that $p\circ q$ is of minimum length with $p\circ q\notin W\cup\ov{(X\cup Y)-W}$.  As $W$ and $\ov{(X\cup Y)-W}$ are both closed, we may assume $p\in W$ and $q\in\ov{(X\cup Y)-W}$.  If $q\in (X\cup Y)-W$, then $p\circ q\in X\cup Y$ as $X$ and $Y$ are both closed.  Otherwise, $q=q^{\pr}\circ q_1$ where $q^{\pr}\in\ov{(X\cup Y)-W},\ q_1\in(X\cup Y)-W$.  By the minimality hypothesis, $p\circ q^{\pr}\in W\cup\ov{(X\cup Y)-W}$.  If $p\circ q^{\pr}\in W$, then $p\circ q=p\circ q^{\pr}\circ q_1\in X\cup Y$ as $X$ and $Y$ are both closed.  If $p\circ q^{\pr}\in\ov{(X\cup Y)-W}$, then so is $p\circ q^{\pr}\circ q_1$.  Either case contradicts the assumption that $p\circ q\notin W\cup\ov{(X\cup Y)-W}$.  Hence, this set is closed.  Using properties of closure operators, we deduce

$$W\cup\ov{(X\cup Y)-W}=\ov{W\cup\ov{(X\cup Y)-W}}=\ov{W\cup((X\cup Y)-W)}=\ov{X\cup Y}.$$

Now assume $\ov{X\cup Y}$ is not co-closed.  Choose $p\in\ov{X\cup Y}$ of minimum length such that $p=q\circ q^{\pr}$ for some paths $q,q^{\pr}$ not in $\ov{X\cup Y}$.  If $p\in X\cup Y$, then either $q\in X\cup Y$ or $q^{\pr}\in X\cup Y$ since both $X$ and $Y$ are co-closed.  Otherwise, there exist $p_1\in X\cup Y,\ p^{\pr}\in\ov{X\cup Y}$ such that $p=p_1\circ p^{\pr}$.

Suppose $p_1$ is a subpath of $q$ and let $r\in\AP$ such that $p_1\circ r=q$.  Then $r\circ q^{\pr}=p^{\pr}$, so either $r\in\ov{X\cup Y}$ or $q^{\pr}\in\ov{X\cup Y}$ by the minimality of $p$.  But if $r$ is in $\ov{X\cup Y}$ then so is $q=p_1\circ r$.  This contradicts the hypothesis on $q$.

Suppose $q$ is a subpath of $p_1$ and let $r\in\AP$ such that $q\circ r=p_1$.  This implies $r\circ p^{\pr}=q^{\pr}$.  Since $X$ and $Y$ are co-closed, either $r\in X\cup Y$ or $q\in X\cup Y$.  But if $r\in X\cup Y$, then $q^{\pr}=r\circ p^{\pr}\in \overline{X\cup Y}$ holds.  This contradicts the hypothesis on $q^{\pr}$.

Hence, $\ov{X\cup Y}$ is co-closed.  Putting this together, we deduce that $W\cup\ov{(X\cup Y)-W}$ is biclosed, establishing (\ref{thm_cu_criteria_2}).

Partially order the set of acyclic paths by inclusion; that is, for $p,q\in\AP$ set $p\preceq q$ if $p$ is a subpath of $q$.

For $p,q\in\AP$, the set $\ov{\{p,q\}}$ is $\{p,q\}$ if they are not composable and is $\{p,q,p\circ q\}$ otherwise.  In either case, it is easy to verify both (\ref{thm_cu_criteria_3}) and (\ref{thm_cu_criteria_4}).
\end{proof}

Using the properties of Theorem \ref{thm_cu_criteria}, we may determine the structure of all polygons in $\Bic(\AP)$.

\begin{corollary}\label{cor_polygon_ap}
Every polygon of $\Bic(\AP)$ is either a square or hexagon as in Figure~\ref{sqpent}.
\end{corollary}

\begin{figure}[h]
\includegraphics[scale=1]{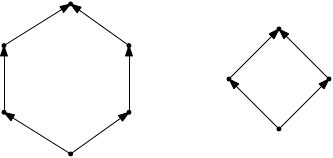}
\caption{The polygons of $\text{Bic}(\text{AP}).$}
\label{sqpent}
\end{figure}

\begin{proof}
A polygon of $\Bic(\AP)$ is an interval of the form $[W,X\vee Y]$ where $W,X,Y\in\Bic(\AP)$ such that $W\lessdot X$ and $W\lessdot Y$. By (1), there exist unique paths $p\in X-W$ and $q\in Y-W$. By (2), $X\vee Y=W\cup\ov{\{p,q\}}$.  If $p$ and $q$ are not composable, then $\ov{\{p,q\}}=\{p,q\}$, which implies that the interval $[W,X\vee Y]$ is a square. Otherwise, $\ov{\{p,q\}}=\{p,q,p\circ q\}$, and the interval $[W,X\vee Y]$ is a hexagon as in Figure~\ref{sqpent}.
\end{proof}

Let $\widehat{Q}$ be the framed quiver of $Q$ with positive \textbf{c}-vectors \textbf{c}-vec$^+(Q)$.  We say that a subset $X$ of \textbf{c}-vec$^+(Q)$ is \textbf{closed} if $x+y\in X$ whenever $x,y\in X$ and $x+y\in$ \textbf{c}-vec$^+(Q)$.

The relation to the previous closure operator is that if $Q$ is of type $\mathbb{A}$ or is an oriented cycle, then the positive \textbf{c}-vectors of $Q$ are in natural bijection with acyclic paths in the underlying graph of $Q$.  Moreover, the closure operators are identified via this bijection. Thus we define $\text{Bic}(Q)$ to be the lattice of \textbf{biclosed sets of \textbf{c}-vectors} of $Q$.

\section{Biclosed subcategories}\label{biclosedsubcat}

Throughout this section, we assume that $\Lambda = \Bbbk Q/I$ is the cluster-tilted algebra defined by a quiver $Q$, which is either a cyclic quiver or of type $\mathbb{A}$. In this section, we show how to translate the information of $\text{Bic}(Q)$ into a lattice of \textbf{biclosed} subcategories of $\Lambda$-$\text{mod}$ that we will denote by $\mathcal{BIC}(Q).$ More specifically, each biclosed set $B \in \text{Bic}(Q)$ will determine a unique subcategory $\mathcal{B}$ of $\Lambda$-mod and an inclusion of biclosed sets $B_1 \subset B_2$ will translate into an inclusion of biclosed subcategories $\mathcal{B}_1 \subset \mathcal{B}_2.$ Using the additional algebraic data that accompanies these subcategories, we prove that the oriented exchange graph defined by $Q$ is a lattice quotient of $\mathcal{BIC}(Q).$

\begin{definition}\label{bicsub}
Let $\mathcal{C}$ be a subcategory of $\Lambda\text{-mod}$. We say that $\mathcal{C}$ is \textbf{biclosed} if 

\begin{itemize}
\item[$i)$] $\mathcal{C} = \text{add}\left(\oplus_{i \in [k]} M(w_i) \right)$ for some set of $\Lambda$-modules $\{M(w_i)\}_{i \in [k]}$; 
\item[$ii)$] $\mathcal{C}$ is \textbf{weakly extension closed} (i.e., if $0 \to M(w_1) \to M(w_3) \to M(w_2) \to 0$ is an exact sequence with $M(w_1), M(w_2) \in \mathcal{C}$, then $M(w_3) \in \mathcal{C}$);
\item[iii)] $\mathcal{C}$ is \textbf{weakly extension co-closed} (i.e., if $0 \to M(w_1) \to M(w_3) \to M(w_2) \to 0$ is an exact sequence with $M(w_1), M(w_2) \not \in \mathcal{C}$, then $M(w_3) \not \in \mathcal{C}$).
\end{itemize}

\noindent Let $\mathcal{BIC}(Q)$ denote the collection of biclosed subcategories of $\Lambda$-mod ordered by inclusion.
\end{definition}

Part $i)$ in Definition~\ref{bicsub} says that the elements of $\mathcal{BIC}(Q)$ are additively generated subcategories $\mathcal{C}$ of $\Lambda\text{-mod}$. Since we are restricting our attention to representation-finite algebras and thus to module categories with finitely many indecomposable objects, we can define \textbf{complementation} on the collection of additively generated subcategories of $\Lambda\text{-mod}$. Let $\mathcal{ADD}(Q)$ denote the collection of additively generated subcategories of $\Lambda\text{-mod}$ and let $A := \{M(w_i)\}_{i \in [k]}$ be any set of indecomposable $\Lambda$-modules. We define the \textbf{complementation} of an additively generated subcategory by
$$\begin{array}{rcl}
\mathcal{ADD}(Q) & \stackrel{(-)^c}{\longrightarrow} & \mathcal{ADD}(Q)\\
\mathcal{A} := \text{add}(\oplus M(w_i): M(w_i) \in A) & \longmapsto & \mathcal{A}^c := \text{add}(\oplus M(w_i): M(w_i) \not \in A).
\end{array}$$
Clearly, $(\mathcal{A}^c)^c = \mathcal{A}.$ It is also clear from the definition of biclosed subcategories of $\Lambda$-mod that complementation restricts to a duality $(-)^c: \mathcal{BIC}(Q) \to \mathcal{BIC}(Q).$ 

Additionally, we remark that the standard duality (i.e., $D(-):=\Hom_\Bbbk(-,\Bbbk)$) gives us the following bijection
$$\begin{array}{rcl}
\mathcal{ADD}(Q) & \stackrel{D(-)}{\longrightarrow} & \mathcal{ADD}(Q^{\text{op}})\\
\mathcal{A} := \text{add}(\oplus M(w_i): M(w_i) \in A) & \longmapsto & D\mathcal{A} := \text{add}(\oplus DM(w_i): M(w_i) \in A).
\end{array}$$
As with complementation, one has $D(D\mathcal{A}) = \mathcal{A}.$ The following obvious lemma shows that the standard duality and complementation interact nicely.

\begin{lemma}
For any $\mathcal{A} \in \mathcal{ADD}(Q)$, we have that $(D\mathcal{A})^c = D(\mathcal{A}^c).$
\end{lemma}

To prove that $\mathcal{BIC}(Q)$ is a lattice, we show that it is isomorphic to $\Bic(Q)$. The key insight is the following lemma, which follows easily from the description of extensions of strings modules of $Q$ in Lemma~\ref{extq_n} and \cite[Theorem 3.7, Corollary 4.4]{cs14}.

\begin{lemma}
  For $i=1,2,3$, let $X_i\in \text{ind}(\Lambda\text{-mod})$ and $\textbf{c}_i=\un{\text{dim}}(X_i)$. Then $\textbf{c}_3=\textbf{c}_1+\textbf{c}_2$ if and only if there exists a short exact sequence $0\ra X_j\ra X_3\ra X_k\ra 0$ with $\{j,k\}=\{1,2\}$.
\end{lemma}

From this lemma, the closure space on $\text{ind}(\Lambda\text{-mod})$ induced by the weak extension closure is isomorphic to the closure space on \textbf{c}-vectors of $Q$. There is a canonical bijection between $\mathcal{ADD}(Q)$ and subsets of $\text{ind}(\Lambda\text{-mod})$, from which we may deduce the following isomorphism.

\begin{proposition}
We have the following isomorphism of posets

$$\begin{array}{rcl}
\text{Bic}(Q) & \stackrel{\sim}{\longrightarrow} & \mathcal{BIC}(Q)\\
B & \longmapsto & \mathcal{B}:= \text{add}(\oplus_{\textbf{c} \in B} M(w({\textbf{c}})))\\
B := \{\underline{\dim} (M(w)) \in \mathbb{Z}^n: \ M(w) \in {\text{ind}(\mathcal{B})}\} & \testleftlong & \mathcal{B}.
\end{array}$$

\noindent In particular, $\mathcal{BIC}(Q)$ is a lattice.
\end{proposition}

Let $\mathcal{A}_1 := \text{add}(\oplus_{i\in[n]} M(w_i)), \mathcal{A}_2 :=\text{add}(\oplus_{j \in [m]} M(v_j)) \in \mathcal{ADD}(Q).$ We define $$\mathcal{A}_1 \cup \mathcal{A}_2 := \text{add}\left(\bigoplus_{i \in [n]}M(w_i) \oplus \bigoplus_{j \in [m]}M(v_j)\right) \in \mathcal{ADD}(Q)$$ and $$\mathcal{A}_1\setminus \mathcal{A}_2 := \text{add}\left(\bigoplus_{i \in [n]}M(w_i) : \ M(w_i) \not \in \{M(v_j)\}_{j \in  [m]}\right) \in \mathcal{ADD}(Q).$$
If $M$ is a module in $\mathcal{A} = \text{add}(\oplus_{i \in [n]} M(w_i)) \in \mathcal{ADD}(Q),$ we define $\mathcal{A}\setminus M$ to be the largest additively generated subcategory of $\mathcal{A}$ not containing any modules that have common nonzero summands with $M$.  Additionally, we define $\overline{\mathcal{A}} \in \mathcal{ADD}(Q)$ to be the smallest additively generated subcategory of $\Lambda$-mod containing $\mathcal{A}$ that is weakly extension closed.

We can now translate the formula for the join of two biclosed sets of \textbf{c}-vectors into a formula for the join of two biclosed subcategories.

\begin{corollary}\label{joininBIC}
If $\mathcal{B}_1, \mathcal{B}_2 \in \mathcal{BIC}(Q),$ then $\mathcal{B}_1 \vee \mathcal{B}_2 = \overline{\mathcal{B}_1 \cup \mathcal{B}_2}.$ 
\end{corollary}

The fact that $\overline{\mathcal{B}_1 \cup \mathcal{B}_2}$ in Corollary~\ref{joininBIC} is co-closed follows from Theorem~\ref{thm_cu_criteria}(\ref{thm_cu_criteria_2}).

\begin{lemma}\label{torsinclude}
There is an inclusion of posets $\text{tors}(\Lambda) \hookrightarrow \mathcal{BIC}(Q).$
\end{lemma}
\begin{proof}
Let $\mathcal{T} \in \text{tors}(\Lambda).$ Since $\Lambda$ is representation-finite and since $\mathcal{T}$ is a torsion class, $\mathcal{T} = \text{add}(\oplus_{i \in [k]} M(w_i))$ for some collection of indecomposables $\{M(w_i)\}_{i \in [k]}$. Since $\mathcal{T}$ is extension closed, it is weakly extension closed.

Assume $0 \to X \to Z \to Y \to 0$ is an exact sequence with $X, Y \not \in \mathcal{T}$. Suppose $Z \in \mathcal{T}$. Then since $\mathcal{T}$ is quotient closed, $Y \in \mathcal{T}$, a contradiction. Thus $\mathcal{T}$ is weakly extension co-closed so $\mathcal{T} \in \mathcal{BIC}(Q).$
\end{proof}

Let $\mathcal{B} \in \mathcal{BIC}(Q)$. Define $\mathbb{X}(\mathcal{B})$ to be the set of objects $X$ of $\mathcal{B}$ up to isomorphism with the property that if one has a surjection $X \twoheadrightarrow Y$ then $Y \in \mathcal{B}$. Observe that $\mathcal{B}$ is quotient closed if and only if $\mathbb{X}(\mathcal{B}) = \mathcal{B}.$ Also, define $\mathbb{Y}(\mathcal{B})$ to be the set of objects $X \in \Lambda\text{-mod}$ up to isomorphism with the property that there exists a nonzero object $Y$ of $\mathcal{B}$ such that $Y \hookrightarrow X.$ Now define maps $\pi_\downarrow, \pi^\uparrow: \mathcal{BIC}(Q) \to \mathcal{ADD}(Q)$ by $$\pi_\downarrow(\mathcal{B}) := \text{add}\left(\displaystyle\bigoplus M : \ M \in \text{ind}(\mathbb{X}(\mathcal{B})) \right)$$ and $$\pi^\uparrow(\mathcal{B}) := \text{add}\left(\displaystyle\bigoplus M : \ M \in \text{ind}(\mathbb{Y}(\mathcal{B})) \right).$$ Clearly, $\pi_\downarrow(\mathcal{B}) \subset \mathcal{B} \subset \pi^\uparrow(\mathcal{B}),$ $\pi_\downarrow \circ \pi_\downarrow = \pi_\downarrow,$ and $\pi^\uparrow \circ \pi^\uparrow = \pi^\uparrow.$

\begin{proposition}\label{pidownintors}
If $\mathcal{B}\in \mathcal{BIC}(Q),$ then $\pi_\downarrow(\mathcal{B}) \in \text{tors}(\Lambda).$ Furthermore, $\pi_\downarrow(\mathcal{BIC}(Q)) = \text{tors}(\Lambda).$ 
\end{proposition}
\begin{proof}
Given $\mathcal{B} \in \mathcal{BIC}(Q)$, Lemma~\ref{quotclosed} shows that $\pi_\downarrow(\mathcal{B})$ is a full, additive, quotient closed subcategory of $\Lambda$-mod. That $\pi_\downarrow(\mathcal{B})$ is extension closed follows from Lemma~\ref{extclosed}. Thus $\pi_\downarrow(\mathcal{B}) \in \text{tors}(\Lambda).$ The second assertion now follows from Lemma~\ref{torsinclude}.
\end{proof}

\begin{theorem}\label{lattquot}
Let $\Theta$ denote the equivalence relation on $\mathcal{BIC}(Q)$ where $\mathcal{B}_1 \equiv \mathcal{B}_2 \text{ mod } \Theta$ if and only if $\pi_\downarrow(\mathcal{B}_1) = \pi_\downarrow(\mathcal{B}_2).$ Then $\pi_\downarrow: \mathcal{BIC}(Q) \to \text{tors}(\Lambda)$ is a lattice quotient map. In particular, $\overrightarrow{EG}(\widehat{Q}) \cong \mathcal{BIC}(Q)/\Theta.$
\end{theorem}
\begin{proof}
We prove this Theorem by appealing to Lemma~\ref{lattconglemma}. By definition, $\pi_\downarrow$ and $\pi^\uparrow$ are idempotent. By Proposition~\ref{pidownintors} and Lemma~\ref{torsinclude}, we know that $\pi_\downarrow(\mathcal{B}) \in \mathcal{BIC}(Q)$ for any $\mathcal{B} \in \mathcal{BIC}(Q).$ By Lemma~\ref{piupprops} $a),$ we have that $\pi^\uparrow(\mathcal{B}) \in \mathcal{BIC}(Q)$ for any $\mathcal{B} \in \mathcal{BIC}(Q).$ By Lemma~\ref{pidownordpres} and Lemma~\ref{piupprops} $b)$, we know that both $\pi_\downarrow$ and $\pi^\uparrow$ are order-preserving. Lastly, by Lemma~\ref{piuppidown} $a)$ and $b)$, we know that $\pi_{\downarrow}\circ\pi^{\uparrow}=\pi_{\downarrow}$ and $\pi^{\uparrow}\circ\pi_{\downarrow}=\pi^{\uparrow}$. By Lemma~\ref{lattconglemma}, we obtain that $\pi_\downarrow$ is a lattice quotient map. The last assertion immediately follows from the fact that $\text{tors}(\Lambda) \cong \overrightarrow{EG}(\widehat{Q}).$
\end{proof}

\begin{corollary}\label{polygonalflips}
Let $Q$ be either a type $\mathbb{A}$ quiver or a cyclic quiver. Then any two maximal green sequences of $Q$ are connected by a sequence of polygonal flips.  Moreover, every polygon in $\overrightarrow{EG}(\widehat{Q})$ is either a square or pentagon (see Figure~\ref{sqpent_tors}).
\end{corollary}

\begin{figure}[h]
\includegraphics[scale=1]{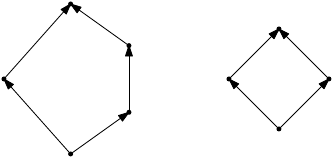}
\caption{The polygons in the oriented exchange graphs from Corollary~\ref{polygonalflips}.}
\label{sqpent_tors}
\end{figure}

\begin{proof}
Since $\mathcal{BIC}(Q)$ is polygonal by Theorem \ref{bicapprops} and polygonality is preserved by lattice quotients (see \cite[Proposition 9-6.9]{ReadingPAB}), Theorem \ref{lattquot} implies that $\overrightarrow{EG}(\widehat{Q})$ is a polygonal lattice. Since maximal green sequences correspond to maximal chains in $\overrightarrow{EG}(\widehat{Q})$, Lemma \ref{lem_poly_flip} implies that any two maximal green sequences are connected by polygonal flips.

Let $\ov{Q},\ov{Q}_1,\ov{Q}_2\in\overrightarrow{EG}(\widehat{Q})$ such that $\ov{Q}_1$ and $\ov{Q}_2$ are distinct ice quivers covering $\ov{Q}$. Let $\Bcal\in\mathcal{BIC}(Q)$ such that $\pi_{\downarrow}(\Bcal)$ is the torsion class corresponding to $\ov{Q}$ and $\pi^{\uparrow}(\Bcal)=\Bcal$. Then there exist $\Bcal_1,\Bcal_2\in\mathcal{BIC}(Q)$ both covering $\Bcal$ such that $\pi_{\downarrow}(\Bcal_i)$ is the torsion class corresponding to $\ov{Q}_i$ for $i=1,2$. Then $[\Bcal,\Bcal_1\vee\Bcal_2]$ is a polygon of $\mathcal{BIC}(Q)$, so it is either a square or hexagon. Restricting $\Theta$ to the interval $[\Bcal,\Bcal_1\vee\Bcal_2]$, the polygon $[\ov{Q},\ov{Q}_1\vee\ov{Q}_2]$ is a lattice quotient of a square or hexagon as in Figure~\ref{sqpent}. Hence, this interval is either a square, pentagon, or hexagon.

Suppose $[\Bcal,\Bcal_1\vee\Bcal_2]$ is a hexagon, and let $M(u_1)$ and $M(u_2)$ be the unique indecomposables in $\Bcal\setminus\Bcal_1$ and $\Bcal\setminus\Bcal_2$, respectively. By the description of polygons in the proof of Corollary \ref{cor_polygon_ap}, there exists an extension, without loss of generality, of the form $0\ra M(u_1)\ra M(w)\ra M(u_2)\ra 0$. Then the covering relation $\Bcal_1\lessdot\Bcal_1\cup\text{add}(M(w))$ is contracted by $\Theta$. Hence, $[\ov{Q},\ov{Q}_1\vee\ov{Q}_2]$ cannot be a hexagon.
\end{proof}

We now address a conjecture on the lengths of maximal green sequences (see \cite[Conjecture 2.22]{bdp}) and give an affirmative answer when $Q$ is a type $\mathbb{A}$ quiver or a cyclic quiver. Let $\text{green}_\ell(Q) :=\{\textbf{i} \in \text{green}(Q): \ \ell en(\textbf{i}) = \ell \}$ be the set of maximal green sequences of length $\ell$.

\begin{corollary}\label{lengths}
Let $Q$ be either a type $\mathbb{A}$ quiver or a cyclic quiver. Then the set $\{\ell \in \mathbb{N}:  \text{green}_\ell(Q) \neq \emptyset\}$ is an interval in $\mathbb{N}$.
\end{corollary}
\begin{proof}
Since $\overrightarrow{EG}(\widehat{Q})$ is nonempty and finite when $Q$ is a type $\mathbb{A}$ quiver or a cyclic quiver, $\text{green}(Q) \neq \emptyset.$ Furthermore, $\overrightarrow{EG}(\widehat{Q})$ is a finite lattice so it has only finitely many maximal chains. Thus $\overrightarrow{EG}(\widehat{Q})$ has only finitely many maximal green sequences. Let $\textbf{i}_{\text{min}}$ (resp., $\textbf{i}_{\text{max}}$) be a maximal green sequence of $Q$ of smallest (resp., largest) length. Let $\ell_{\text{min}}:= \ell en(\textbf{i}_{\text{min}})$ and $\ell_{\text{max}}:= \ell en(\textbf{i}_{\text{max}})$. By Lemma~\ref{lem_poly_flip} and by regarding maximal green sequences as maximal chains in $\overrightarrow{EG}(\widehat{Q})$, there exists maximal green sequences $\textbf{i}_{\text{min}} = \textbf{i}_0, \textbf{i}_1, \ldots, \textbf{i}_k = \textbf{i}_{\text{max}}$ where $\textbf{i}_j \in \text{green}(Q)$ and where $\textbf{i}_j$ and $\textbf{i}_{j+1}$ differ by a polygonal flip for all $j$. By Corollary~\ref{polygonalflips}, $|\ell en(\textbf{i}_j) - \ell en(\textbf{i}_{j-1})| \le 1$ for each $j \in [k].$ Thus for each $\ell \in [\ell_{\text{min}},\ell_{\text{max}}]$ there exists $\textbf{i} \in \text{green}(Q)$ such that $\ell en(\textbf{i}) = \ell$.
\end{proof}

\begin{remark}
In \cite{kase}, it is shown that if $Q$ is a path quiver, then $\{\ell \in \mathbb{N}:  \text{green}_\ell(Q) \neq \emptyset\} = \left[n, \frac{n(n+1)}{2}\right].$ Note that here $ \frac{n(n+1)}{2} = |\textbf{c}\text{-vec}^+(Q)|.$ If $Q$ is mutation-equivalent to a path quiver, we only know that $\{\ell \in \mathbb{N}:  \text{green}_\ell(Q) \neq \emptyset\}$ is an interval in $\mathbb{N}$ that is contained in $\left[n, \frac{n(n+1)}{2}\right].$ For example, if $Q$ is the cyclic quiver appearing in Example~\ref{stringsA3}, then its maximal green sequences are of length 4 or 5.
\end{remark}

\begin{example}
Let $Q = Q(3)$ and let $\Lambda = \Bbbk Q(3)/\langle \alpha_1\alpha_2: \ \alpha_i \in Q(3)_1\rangle.$ In Figure~\ref{bicq3}, we show how $\pi_\downarrow$ maps elements of $\mathcal{BIC}(Q(3))$ to elements of $\text{tors}(\Lambda)$ using the notation in Example~\ref{torsexample} for additively generated subcatgories of $\Lambda$-mod. For instance, $\text{add}\left(X(3,2)\oplus X(2,1) \oplus X(2,2)\right) \in \mathcal{BIC}(Q(3))$ is represented as $\ \includegraphics[scale=1.25]{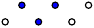}.$ As in Figure~\ref{fig_lattice_quotient}, blue edges of $\mathcal{BIC}(Q(3))$ indicate edges that will be contracted to form $\text{tors}(\Lambda).$

\begin{figure}[h]
\includegraphics[scale=1.3]{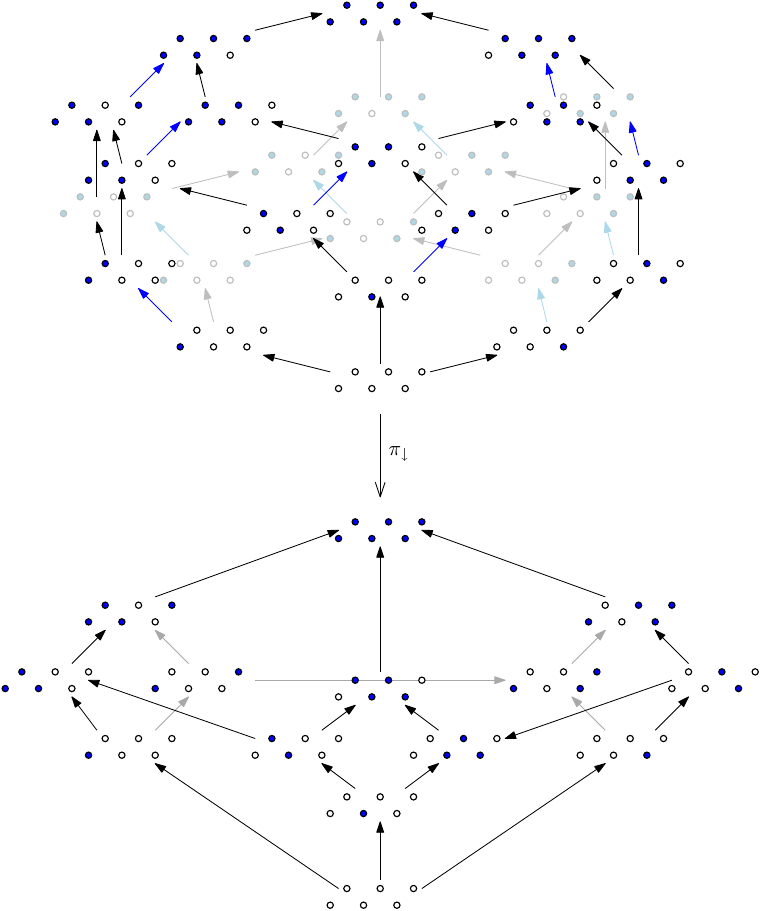}
\caption{The map $\pi_\downarrow: \mathcal{BIC}(Q(3)) \to \text{tors}(\Lambda).$}
\label{bicq3}
\end{figure}

\end{example}

\section{Properties of $\pi_\downarrow$ and $\pi^\uparrow$}\label{aboutpiupanddown}

In this section, we prove several lemmas that establish important properties satisfied by $\pi_\downarrow$ and $\pi^\uparrow$. Throughout this section, we assume that $\Lambda = \Bbbk Q/I$ is the cluster-tilted algebra defined by a quiver $Q$, which is either a cyclic quiver or of type $\mathbb{A}$. Before presenting these lemmas and their proofs, we introduce some additional notation for string modules. Let $M(w) \in \text{ind}(\Lambda\text{-mod})$ be a string module with $$w = x_1 \stackrel{\alpha_1}{\longleftrightarrow} x_2 \cdots x_{i} \stackrel{\alpha_i}{\longleftrightarrow} x_{i+1} \cdots x_m \stackrel{\alpha_{m}}{\longleftrightarrow} x_{m+1}. $$ Define $\text{Pred}(\alpha_i) := x_1 \stackrel{\alpha_1}{\longleftrightarrow} x_2 \cdots x_{i-1} \stackrel{\alpha_{i-1}}{\longleftrightarrow} x_{i}$ and $\text{Succ}(\alpha_i) := x_{i+1} \stackrel{\alpha_{i+1}}{\longleftrightarrow} x_{i+2} \cdots x_{m} \stackrel{\alpha_{m}}{\longleftrightarrow} x_{m+1}.$ 

\begin{lemma}\label{quotclosed}
If $\mathcal{B} \in \mathcal{BIC}(Q),$ then $\pi_\downarrow(\mathcal{B})$ is a full, additive, quotient closed subcategory of $\Lambda$-mod.
\end{lemma}
\begin{proof}
By the definition of $\pi_\downarrow(\mathcal{B})$, it is clear that $\pi_\downarrow(\mathcal{B})$ is a full, additive subcategory of $\Lambda$-mod.

Next, we show that $\pi_\downarrow(\mathcal{B})$ is quotient closed. Suppose that we have a surjection $M(w_1) \twoheadrightarrow M(w_2)$ where $M(w_1) \in \pi_\downarrow(\mathcal{B})$. Now if there exists a surjection $M(w_2) \twoheadrightarrow Y$ for some $Y \in \Lambda\text{-mod}$, then we have a surjection $M(w_1) \twoheadrightarrow Y$. Thus $Y \in \mathcal{B}$ so $M(w_2) \in \pi_\downarrow(\mathcal{B})$. 

Now suppose $X \in \pi_{\downarrow}(\mathcal{B})$ and we have a surjection $f: X = \oplus_{i \in [k]}M(w^{(i)})^{a_i} \twoheadrightarrow M(v)$ for some integers $a_i \ge 0$ and no summand $M(w^{(i)})$ surjects onto $M(v)$. Furthermore, suppose that an indecomposable $M(u)$ belongs to $\mathcal{B}$ if $\text{dim}_\Bbbk(M(u)) < \text{dim}_\Bbbk(M(v))$ and has the property that an object of $\pi_{\downarrow}(\mathcal{B})$ surjects onto $M(u)$. Let $M(w^{(i)})$ be a summand of $X$ where the component map $g: M(w^{(i)}) \to M(v)$ of $f$ is nonzero. By Lemma~\ref{dimhom}, there exists a nonempty string $u$ such that $M(w^{(i)}) \twoheadrightarrow M(u) \hookrightarrow M(v)$. This implies that $M(u) \in \mathcal{B}.$

We can now write $v = v^{(1)} \rightarrow u \leftarrow v^{(2)}$ where $v^{(1)}$ or $v^{(2)}$ is a nonempty string. Since $f$ is a surjection, we obtain surjections $X \twoheadrightarrow M(v^{(j)})$ for $j = 1, 2$. As $\text{dim}_\Bbbk(M(v^{(j)})) < \text{dim}_\Bbbk(M(v))$ for $j = 1, 2$, we see that $M(v^{(1)}), M(v^{(2)}) \in \mathcal{B}$. Since $\mathcal{B}$ is weakly extension closed, we obtain that $M(v) \in \mathcal{B}$. Consequently, for any surjection $X \twoheadrightarrow Y$ where $X \in \pi_\downarrow(\mathcal{B})$ and any $Y \in \Lambda\text{-mod}$, we have that $Y \in \mathcal{B}$.

Finally, assume we have $X \twoheadrightarrow Y$ and $Y \twoheadrightarrow Z$ where $X \in \pi_\downarrow(\mathcal{B})$ and $Y, Z \in \Lambda\text{-mod}$. Composing these surjections produces a surjection $X \twoheadrightarrow Z$. From the previous paragraph, we have that $Z \in \mathcal{B}$. Thus $Y \in \pi_\downarrow(\mathcal{B}).$
\end{proof}

\begin{lemma}\label{weakext}
If $\mathcal{B} \in \mathcal{BIC}(Q),$ then $\pi_\downarrow(\mathcal{B})$ is weakly extension closed.
\end{lemma}
\begin{proof}
Let $0 \to M(w_2) \to Z \to M(w_1) \to 0$ be an extension where $M(w_2), M(w_1) \in \pi_\downarrow(\mathcal{B})$ and where $Z \in \text{ind}(\mathcal{B}).$ Then $Z \in \text{ind}(\mathcal{B})$ since $\mathcal{B}$ is weakly extension closed. It is easy to see that $Z = M(w_2 \stackrel{\alpha}{\longleftarrow} w_1).$ Let $w_2 = x^{(2)}_1 \stackrel{\alpha^{(2)}_1}{\longleftrightarrow} x^{(2)}_2 \cdots x^{(2)}_{n_2-1} \stackrel{\alpha^{(2)}_{n_2-1}}{\longleftrightarrow} x^{(2)}_{n_2}$ and let $w_1 = x^{(1)}_1 \stackrel{\alpha^{(1)}_1}{\longleftrightarrow} x^{(1)}_2 \cdots x^{(1)}_{n_1-1} \stackrel{\alpha^{(1)}_{n_1-1}}{\longleftrightarrow} x^{(1)}_{n_1}.$

To show that $\pi_\downarrow(\mathcal{B})$ is weakly extension closed we must show that for any surjection $p: M(w_2 \stackrel{\alpha}{\longleftarrow} w_1) \twoheadrightarrow M(u)$ one has $M(u) \in \mathcal{B}.$ Suppose we have such a surjection and suppose that $u$ is substring of $w_1.$ Then we have $$u = x^{(1)}_i \stackrel{\alpha^{(1)}_i}{\longleftrightarrow} x^{(1)}_{i+1} \cdots x^{(1)}_{j-1} \stackrel{\alpha^{(1)}_{j-1}}{\longleftrightarrow} x^{(1)}_{j}$$ where $i, j \in [n_1]$ and $i \le j$. Let $\beta \in Q_1$ be an arrow that appears in $w_2 \stackrel{\alpha}{\longleftarrow} w_1$ with exactly one of its vertices belonging to $u$. Such an arrow $\beta$ belongs to the set $\{\alpha, \alpha^{(1)}_1, \ldots, \alpha_{n_1-1}^{(1)}\}$. Since $p$ is a surjection, the unique vertex of $\beta$ that belongs to $u$ is the source of $\beta$. Thus we have that $M(w_1) \twoheadrightarrow M(u).$ Therefore, $M(u) \in \pi_\downarrow(\mathcal{B})$ by Lemma~\ref{quotclosed} so $M(u) \in \mathcal{B}$. An analogous proof shows that $M(u) \in \mathcal{B}$ if $u$ is a substring of $M(w_2).$

To complete the proof, we need to show that if $p: M(w_2 \stackrel{\alpha}{\longleftarrow} w_1) \twoheadrightarrow M(u)$ for some string $u = x^{(2)}_i \stackrel{}{\longleftrightarrow} \cdots  \stackrel{}{\longleftrightarrow} x^{(1)}_{j}$ with $i \in [n_2]$ and $j \in [n_1],$ then $M(u) \in \mathcal{B}.$ Suppose to the contrary that $M(w_2 \stackrel{\alpha}{\longleftarrow} w_1)$ is of minimal dimension with the property that $p: M(w_2 \stackrel{\alpha}{\longleftarrow} w_1) \twoheadrightarrow M(u)$ is a surjection where $u$ is a string of the above form, but $M(u) \not \in \mathcal{B}$. Since $M(w_2\stackrel{\alpha}{\longleftarrow}w_1) \in \mathcal{B},$ we know that $\dim_\Bbbk(M(u)) < \dim_\Bbbk(M(w_2 \stackrel{\alpha}{\longleftarrow} w_1)).$ Thus $i \neq 1$ or $j \neq n_1.$ We will assume that $i \neq 1$ and $j \neq n_1$ and the proof in the case where exactly one of these conditions is satisfied is analogous. 

Now, since $p$ is a surjection, we know that $ i \neq 1$ implies that $s(\alpha^{(2)}_{i-1}) = x^{(2)}_i$ and $t(\alpha^{(2)}_{i-1}) = x^{(2)}_{i-1}.$ Similarly, $j \neq n_1$ implies that $s(\alpha^{(1)}_j) = x_j^{(1)}$ and $t(\alpha^{(1)}_{j}) = x^{(1)}_{j+1}.$ Observe that we have the exact sequence $$0 \to M(x_1^{(2)} \stackrel{}{\leftrightarrow} \cdots \stackrel{}{\leftrightarrow} x^{(2)}_{i-1})\oplus M(x_{j+1}^{(1)} \stackrel{}{\leftrightarrow} \cdots \stackrel{}{\leftrightarrow} x^{(1)}_{n_1}) \to M(w_2 \stackrel{\alpha}{\longleftarrow} w_1) \to M(u) \to 0.$$ From these facts, we deduce that we have the following two exact sequences $$0 \to M(x_1^{(2)} \stackrel{}{\leftrightarrow} \cdots \stackrel{}{\leftrightarrow} x^{(2)}_{i-1}) \to M(w_2) \to M(x^{(2)}_i \leftrightarrow \cdots \leftrightarrow x_{n_2}^{(2)}) \to 0$$ $$0 \to M(x_{j+1}^{(1)} \stackrel{}{\leftrightarrow} \cdots \stackrel{}{\leftrightarrow} x^{(1)}_{n_1}) \to M(w_1) \to M(x^{(1)}_1 \leftrightarrow \cdots \leftrightarrow x_{j}^{(1)}) \to 0.$$ By Lemma~\ref{quotclosed}, we have that $M(x^{(2)}_i \leftrightarrow \cdots \leftrightarrow x_{n_2}^{(2)}), M(x^{(1)}_1 \leftrightarrow \cdots \leftrightarrow x_{j}^{(1)}) \in \pi_\downarrow(\mathcal{B}).$ Now notice that we have the exact sequence $$0 \to M(x^{(2)}_i \leftrightarrow \cdots \leftrightarrow x_{n_2}^{(2)}) \to M(u) \to M(x^{(1)}_1 \leftrightarrow \cdots \leftrightarrow x_{j}^{(1)}) \to 0.$$ Since $\dim_\Bbbk(M(u)) < \dim_\Bbbk(M(w_2 \stackrel{\alpha}{\longleftarrow} w_1))$ and since $M(w_2 \stackrel{\alpha}{\longleftarrow} w_1)$ was a counterexample of minimal dimension, we have that $M(u) \in \mathcal{B},$ a contradiction.

We conclude that $M(w_2 \stackrel{\alpha}{\longleftarrow} w_1)\in \pi_\downarrow(\mathcal{B}).$ Thus $\pi_\downarrow(\mathcal{B})$ is weakly extension closed.
\end{proof}

\begin{lemma}\label{extclosed}
If $\mathcal{B} \in \mathcal{BIC}(Q),$ then $\pi_\downarrow(\mathcal{B})$ is extension closed.
\end{lemma}
\begin{proof}
{Consider the subcategory $\mathcal{F}ilt(\pi_\downarrow(\mathcal{B})) \subset \Lambda$-mod as defined in Section~\ref{Sec_42}. Clearly, $\pi_\downarrow(\mathcal{B}) \subset \mathcal{F}ilt(\pi_\downarrow(\mathcal{B}))$. By Lemma~\ref{quotclosed}, we know that $\pi_\downarrow(\mathcal{B})$ is quotient closed. Thus, as is observed in \cite[Proposition 3.3]{dij}, it is known that $\mathcal{F}ilt(\pi_\downarrow(\mathcal{B}))$ is a torsion class and, therefore, is extension closed. We show that $\pi_\downarrow(\mathcal{B}) = \mathcal{F}ilt(\pi_\downarrow(\mathcal{B})).$}

{To complete the proof, we show that $\mathcal{F}ilt(\pi_\downarrow(\mathcal{B})) \subset \pi_\downarrow(\mathcal{B})$. We show that any indecomposable object of $\mathcal{F}ilt(\pi_\downarrow(\mathcal{B}))$ belongs to $\pi_\downarrow(\mathcal{B})$. Let $M(w) \in \text{ind}(\mathcal{F}ilt(\pi_\downarrow(\mathcal{B})))$ and let $0 = M_0 \subseteq M_1 \subset \cdots \subset M_k = M(w)$ be a filtration witnessing that $M(w) \in \mathcal{F}ilt(\pi_\downarrow(\mathcal{B}))$. Furthermore, we assume that this filtration is long enough that $M_i/M_{i+1} \in \text{ind}(\pi_\downarrow(\mathcal{B}))$ for all $i \in \{1,\ldots, k\}$. We show by induction that $M_i \in \pi_\downarrow(\mathcal{B})$ for each $i \in \{1,\ldots, k\}$. This is obvious for $M_1$ so we assume that $M_i \in \pi_\downarrow(\mathcal{B})$ and prove that $M_{i+1} \in \pi_\downarrow(\mathcal{B})$.}

{Since $M_i, M_{i+1}/M_i \in \pi_\downarrow(\mathcal{B})$, write $M_i = \oplus_{r=1}^\ell M(u^{(r)})$ and $M_{i+1}/M_i = \oplus_{t=1}^m M(v^{(t)})$ where $M(u^{(r)}), M(v^{(t)}) \in \text{ind}(\pi_\downarrow(\mathcal{B}))$ for all $r \in \{1, \ldots, \ell\}$ and all $t \in \{1,\ldots, m\}$. Moreover, since $M_i$ is submodule of $M(w)$, we have that $\text{supp}(M(u^{(r)}))\cap \text{supp}(M(u^{(r^\prime)})) = \emptyset$ for any distinct $r, r^\prime \in \{1,\ldots, \ell\}$. Similarly, $\text{supp}(M(u^{(t)}))\cap \text{supp}(M(u^{(t^\prime)})) = \emptyset$ for any distinct $t, t^\prime \in \{1,\ldots, m\}$. Since $0 \to \oplus_{r=1}^\ell M(u^{(r)}) \to M_{i+1} \to \oplus_{t=1}^m M(v^{(t)})\to 0$ is an extension and $M_{i+1}$ is a submodule of $M(w)$, we see that $M_{i+1}$ is a direct sum of submodules $M(w^\prime) \subset M(w)$ with pairwise disjoint supports. Furthermore, each string module $M(w^\prime)$ belongs to $\overline{\{M({u^{(r)}}), M({v^{(t)}})\}}_{1\le r \le \ell, 1\le t \le m}$. Since $\pi_\downarrow(\mathcal{B})$ is weakly extension closed by Lemma~\ref{weakext}, each summand of $M_{i+1}$ belongs to $\pi_\downarrow(\mathcal{B})$. We obtain that $M_{i+1} \in \pi_\downarrow(\mathcal{B})$. Thus, $\pi_\downarrow(\mathcal{B}) = \mathcal{F}ilt(\pi_\downarrow(\mathcal{B}))$.} 
\end{proof}

\begin{lemma}\label{pidownordpres}
The map $\pi_\downarrow: \mathcal{BIC}(Q) \to \text{tors}(\Lambda)$ is order-preserving.
\end{lemma}
\begin{proof}
Let $\mathcal{B}, \mathcal{B}^\prime \in \mathcal{BIC}(Q)$ where $\mathcal{B} \subset \mathcal{B}^\prime.$ Let $X \in \text{ind}(\mathbb{X}(\mathcal{B}))$ and let $X \twoheadrightarrow Y$ be a surjection. Then $Y \in \mathcal{B} \subset \mathcal{B}^\prime$ so $X \in \text{ind}(\mathbb{X}(\mathcal{B}^\prime)).$ Thus $\pi_\downarrow(\mathcal{B}) \subset \pi_\downarrow(\mathcal{B}^\prime)$ so $\pi_\downarrow: \mathcal{BIC}(Q) \to \text{tors}(\Lambda)$ is order-preserving.
\end{proof}

\begin{lemma}\label{duality}
The maps $\pi_\downarrow$ and $\pi^\uparrow$ satisfy $D\pi^\uparrow(\mathcal{B})^c = \pi_\downarrow(D\mathcal{B}^c)$ for any $\mathcal{B} \in \mathcal{BIC}(Q).$
\end{lemma}
\begin{proof}
We have that $$\begin{array}{rcl}
D\pi^\uparrow(\mathcal{B})^c & = &  D(\text{add}(\bigoplus X: X \in \text{ind}(\Lambda\text{-mod}) \text{ where } \exists Y \in \mathcal{B} \text{ such that } Y \hookrightarrow X))^c\\
& = & D\left(\text{add}(\bigoplus X: X \in \text{ind}{(\mathcal{B}^c)} \text{ where }  Y\hookrightarrow X \implies Y \in \mathcal{B}^c)\right) \\
& = & D\left(\text{add}(\bigoplus DX: DX \in \text{ind}(D\mathcal{B}^c) \text{ where } DX \twoheadrightarrow DY \implies DY \in D\mathcal{B}^c)\right)\\
& = & \pi_\downarrow(D\mathcal{B}^c).
\end{array}$$
We remark that in the penultimate line of the calculation $D\mathcal{B}^c \in \mathcal{BIC}(Q^\text{op})$.
\end{proof}

\begin{lemma}\label{piupprops}
The map $\pi^\uparrow: \mathcal{BIC}(Q) \to \mathcal{ADD}(Q)$ satisfies the following:

$\begin{array}{rl}
a) & \pi^\uparrow(\mathcal{B}) \in \mathcal{BIC}(Q),\\
b) & \pi^\uparrow \text{ is order-preserving.}
\end{array}$
\end{lemma}
\begin{proof}
To prove both $a)$ and $b)$, we use that $\pi^\uparrow(\mathcal{B}) =  D(\pi_\downarrow(D\mathcal{B}^c))^c$ for any $\mathcal{B} \in \mathcal{BIC}(Q)$, which follows from Lemma~\ref{duality}. Let $\mathcal{B} \in \mathcal{BIC}(Q)$, then we have that $D\mathcal{B}^c \in \mathcal{BIC}(Q^\text{op}).$ By Proposition~\ref{pidownintors} and Lemma~\ref{torsinclude}, we have that $\pi_\downarrow(D\mathcal{B}^c) \in \mathcal{BIC}(Q^\text{op}).$ Now it follows that $\pi^\uparrow(\mathcal{B}) = D(\pi_\downarrow(D\mathcal{B}^c))^c \in \mathcal{BIC}(Q).$

To prove $b)$, let $\mathcal{B}_1, \mathcal{B}_2 \in \mathcal{BIC}(Q).$ Then we have $$\begin{array}{rcll}
\mathcal{B}_1 \subset \mathcal{B}_2 & \implies & \mathcal{B}_1^c \supset \mathcal{B}_2^c\\
& \implies & D\mathcal{B}_1^c \subset D\mathcal{B}_2^c \\
& \implies & \pi_\downarrow(D\mathcal{B}_1^c) \subset \pi_\downarrow(D\mathcal{B}_2^c) & \text{(by Lemma~\ref{pidownordpres})}\\
& \implies & (\pi_\downarrow(D\mathcal{B}_1^c))^c \supset (\pi_\downarrow(D\mathcal{B}_2^c))^c\\
& \implies & D(\pi_\downarrow(D\mathcal{B}_1^c))^c \subset D(\pi_\downarrow(D\mathcal{B}_2^c))^c\\
& \implies & \pi^\uparrow(\mathcal{B}_1) \subset \pi^\uparrow(\mathcal{B}_2).
\end{array}$$
Thus $\pi^\uparrow$ is order-preserving. \end{proof}

\begin{lemma}\label{piuppidown}
The maps $\pi_\downarrow$ and $\pi^\uparrow$ satisfy the following: 

$\begin{array}{rl}
a) & \pi_\downarrow(\mathcal{B}) = (\pi_\downarrow\circ \pi^\uparrow)(\mathcal{B}) \text{ for any $\mathcal{B} \in \mathcal{BIC}(Q)$},\\
b) & \pi^\uparrow(\mathcal{B}) = (\pi^\uparrow\circ\pi_\downarrow)(\mathcal{B}) \text{ for any $\mathcal{B} \in \mathcal{BIC}(Q)$}.
\end{array}$
\end{lemma}
\begin{proof}
We first prove $a)$. Since $\mathcal{B}\subset \pi^\uparrow(\mathcal{B})$, by Lemma~\ref{pidownordpres} we know that   $\pi_\downarrow(\mathcal{B}) \subset \pi_\downarrow(\pi^\uparrow(\mathcal{B}))$. Thus we need to show that $\pi_\downarrow(\pi^\uparrow(\mathcal{B})) \subset \pi_\downarrow(\mathcal{B}).$ To do so, let $M(u) \in \text{ind}(\mathbb{X}(\pi^\uparrow(\mathcal{B})))$ and suppose that $M(u) \twoheadrightarrow M(w)$ is a surjection where $M(w) \not \in \mathcal{B}$ such that any other such indecomposable $M(w^\prime)$ with $\dim(M(w^\prime)) < \dim(M(w))$ belongs to $\mathcal{B}.$ 

Since $M(w) \in \pi^\uparrow(\mathcal{B}),$ there exists $M(w_1) \in \mathcal{B}$ and an inclusion $M(w_1) \hookrightarrow M(w).$ This inclusion gives rise to an exact sequence $$ 0\to M(w_1) \to M(w) \to M(w)/M(w_1) \to 0.$$ Note that $\dim(M(w)/M(w_1)) < \dim(M(w))$ and we have a surjection $M(u) \twoheadrightarrow M(w) \twoheadrightarrow M(w)/M(w_1)$ so by assumption $M(w)/M(w_1) \in \mathcal{B}.$ If $M(w)/M(w_1)$ is indecomposable, then by the fact that $\mathcal{B}$ is biclosed, $M(w) \in \mathcal{B},$ a contradiction. Thus we can assume $M(w)/M(w_1)$ is not indecomposable.

Observe that since $M(w_1)$ is indecomposable and since $w_1$ is a substring of $w$, we have that $M(w)/M(w_1) = M(w_2)\oplus M(w_3)$ for some substrings of $w$, denoted $w_2$ and $w_3.$ Now observe that we obtain an exact sequence $$0 \to M(w_1) \to M(w_1\leftarrow w_2) \to M(w_2) \to 0.$$ Since $M(w)/M(w_1) \in \mathcal{B},$ we know that $M(w_2), M(w_3) \in \text{ind}(\mathcal{B})$. By the fact that $\mathcal{B}$ is biclosed, we have that $M(w_1\leftarrow w_2) \in \text{ind}(\mathcal{B}).$ We now notice that $w = w_3 \rightarrow w_1 \leftarrow w_2$ so $M(w_1\leftarrow w_2) \hookrightarrow M(w)$ and thus we have the exact sequence $$0 \to M(w_1\leftarrow w_2) \to M(w) \to M(w_3) \to 0.$$ Now by the fact that $\mathcal{B}$ is biclosed, we obtain that $M(w) \in \mathcal{B},$ a contradiction. Thus $M(u) \in \text{ind}(\mathbb{X}(\mathcal{B}))$ and so $\pi_\downarrow(\pi^\uparrow(\mathcal{B})) \subset \pi_\downarrow(\mathcal{B}).$

To prove $b)$, observe that

$$\begin{array}{rclll}
D(\pi^\uparrow(\mathcal{B}))^c & = & \pi_\downarrow(D\mathcal{B}^c) & & \text{(by Lemma~\ref{duality})}\\
& = & \pi_\downarrow(\pi^\uparrow(D\mathcal{B}^c)) & & \text{(by $a)$)} \\
& = & \pi_\downarrow \left(D(\pi_\downarrow(\mathcal{B}))^c\right) & & \text{(by Lemma~\ref{duality})} \\
& = & D(\pi^\uparrow(\pi_\downarrow(\mathcal{B})))^c & & \text{(by Lemma~\ref{duality}).}
\end{array}$$
Thus we have that $\pi^\uparrow(\mathcal{B}) = (\pi^\uparrow\circ\pi_\downarrow)(\mathcal{B}).$ 
\end{proof}

\section{Canonical join-representations}\label{canonical}

In this section, we use our previous results to classify canonical join- and canonical meet-representations of torsion classes $\mathcal{T} \in \text{tors}(\Lambda).$ Throughout this section, we assume that $\Lambda = \Bbbk Q/I$ is the cluster-tilted algebra defined by a quiver $Q$, which is either a cyclic quiver or of type $\mathbb{A}$.

\begin{lemma}\label{facmw}
Let $M(w) \in \text{ind}(\Lambda\text{-mod}).$ Then 

\begin{itemize}
\item[a)] there are no extensions of the form $0\to M(w_1) \to M(v) \to M(w_2) \to 0$ where $M(w_i) \in \text{Fac}(M(w))$ for $i =1,2,$ and
\item[b)] $\text{Fac}(M(w)) \in \text{tors}(\Lambda)$.
\end{itemize}
\end{lemma}
\begin{proof}
$a)$ Suppose we have an extension $0\to M(w_1) \to M(v) \to M(w_2) \to 0$ where $M(w_i) \in \text{Fac}(M(w))$ for $i = 1,2.$ Then by exactness one has that $v = w_1 \stackrel{\alpha}{\longleftarrow} w_2.$ Since $M(w_i) \in \text{Fac}(M(w))$ and since $M(w_i)$ is indecomposable, $M(w) \twoheadrightarrow M(w_i).$ Moreover, $v = w_1 \stackrel{\alpha}{\longleftarrow} w_2$ must be a substring of $w$. However, the orientation of $\alpha$ contradicts that $\Hom_\Lambda(M(w),M(w_1)) \neq 0.$

$b)$ We observe that since $\text{Fac}(M(w))$ is quotient closed, one has $\pi_\downarrow(\text{Fac}(M(w))) = \text{Fac}(M(w)).$ Also, $$\text{Fac}(M(w)) = \text{add}(\oplus M(v_i):\ \exists M(w) \twoheadrightarrow M(v_i))$$ so $\text{Fac}(M(w))$ is additively generated. Thus, by Lemma~\ref{torsinclude}, it remains to show that $\text{Fac}(M(w)) \in \mathcal{BIC}(Q).$ By part $a)$, $\text{Fac}(M(w))$ vacuously is weakly extension closed. Since $\text{Fac}(M(w))$ is quotient closed, any extension $0 \to M(w_1) \to M(v) \to M(w_2) \to 0$ with $M(v) \in \text{Fac}(M(w))$ has $M(w_2) \in \text{Fac}(M(w))$. This means there are no extensions of the form $0 \to M(w_1) \to M(w) \to M(w_2) \to 0$ with $M(w_i) \not \in \text{Fac}(M(w))$ for $i = 1,2.$ Thus $\text{Fac}(M(w))$ is weakly extension co-closed. We conclude that $\text{Fac}(M(w)) \in \mathcal{BIC}(Q).$ \end{proof} 

\begin{remark}
An alternative proof of Lemma~\ref{facmw} $b)$ is obtained by using the fact that indecomposable modules over cluster-tilted algebras are $\tau$-rigid and then applying \cite[Theorem 5.10 (b)]{auslander1981almost}.
\end{remark}

\begin{lemma}
A torsion class $\mathcal{T} \in \text{tors}(\Lambda)$ is join-irreducible if and only if $\mathcal{T} = \text{Fac}(M(w))$ for some $M(w) \in \text{ind}(\Lambda\text{-mod}).$
\end{lemma}
\begin{proof}
Suppose $\mathcal{T} = \text{Fac}(M(w))$ for some $M(w) \in \text{ind}(\Lambda\text{-mod}).$  Let $\mathcal{T}_1, \mathcal{T}_2 \subsetneq \text{Fac}(M(w))$ be torsion classes covered by $\text{Fac}(M(w))$.  This implies that $M(w) \not \in \mathcal{T}_1$ and $M(w) \not \in \mathcal{T}_2.$ Regarding $\mathcal{T}_1$ and $\mathcal{T}_2$ as elements of $\mathcal{BIC}(Q)$, we have that $\mathcal{T}_1 \vee \mathcal{T}_2 = \overline{\mathcal{T}_1 \cup \mathcal{T}_2}$ using Corollary~\ref{joininBIC}.  By Lemma~\ref{facmw} $a)$, every additively generated subcategory $\mathcal{A} \subset \text{Fac}(M(w))$ is weakly extension closed. We therefore have that $\mathcal{T}_1\vee \mathcal{T}_2 = \mathcal{T}_1 \cup \mathcal{T}_2.$ Thus $M(w) \not \in \mathcal{T}_1 \vee \mathcal{T}_2$ so $\mathcal{T}_1 \vee \mathcal{T}_2 \subsetneq \text{Fac}(M(w)),$ a contradiction.

Conversely, suppose that $\mathcal{T} \in \text{tors}(\Lambda)$ is join-irreducible. Since $\text{tors}(\Lambda) = \text{f-tors}(\Lambda)$, we have that $ \mathcal{T} = \text{Fac}(X)$ for some $X \in \Lambda\text{-mod}.$ Let $X = \oplus_{i = 1}^\ell M(w_i)^{a_i}$ for some positive integers $a_i \in \mathbb{N}.$ 

We claim that $\text{Fac}(X) = \bigvee_{i = 1}^\ell \text{Fac}(M(w_i)).$ Observe that for any $j \in [\ell]$, we have $M(w_j) \in \bigvee_{i=1}^\ell \text{Fac}(M(w_i)).$ Since $\bigvee_{i=1}^\ell \text{Fac}(M(w_i)) \in \text{tors}(\Lambda),$ it is additive and thus $X \in \bigvee_{i=1}^\ell \text{Fac}(M(w_i)).$ We conclude that $\text{Fac}(X) \subset \bigvee_{i=1}^\ell \text{Fac}(M(w_i)).$ On the other hand, $M(w_j) \in \text{Fac}(X)$ for any $j \in [\ell]$ since $\text{Fac}(X)$ is quotient closed so we have that $\text{Fac}(M(w_j)) \subset \text{Fac}(X)$ for any $j \in [\ell].$ Since $\text{Fac}(X) \in \text{tors}(\Lambda),$ we have that $\bigvee_{i=1}^\ell \text{Fac}(M(w_i)) \subset \text{Fac}(X).$

Since $\mathcal{T} = \bigvee_{i = 1}^\ell \text{Fac}(M(w_i))$ and $\mathcal{T}$ is join-irreducible, we know that $\mathcal{T} = \text{Fac}(M(w_j))$ for some $j \in [\ell]$. \end{proof}

\begin{theorem}\label{canonjoinrepn}
Let $\mathcal{T} \in \text{tors}(\Lambda).$  Let $M(w_1),\ldots,M(w_\ell)$ be a maximal collection of non-isomorphic indecomposables such that for all $i \in [\ell]$,
\begin{enumerate}
\item $M(w_i)$ is in $\Tcal$ and no proper nonzero submodule of $M(w_i)$ is in $\Tcal$, and
\item if $M(w_i)$ is a proper quotient of $M(w)\in\Tcal$ then $M(w)$ has a proper submodule $M(u)$ in $\Tcal$.
\end{enumerate}
Then $\Tcal=\bigvee_{i=1}^\ell \text{Fac}(M(w_i))$ is a canonical join-representation of $\Tcal$.
\end{theorem}

\begin{proof}
We first prove that the equality $\Tcal=\bigvee_{i=1}^\ell \text{Fac}(M(w_i))$ holds.  Since $\text{Fac}(M(w_i))\subseteq\Tcal$ for all $i$, it is clear that $\Tcal$ contains $\bigvee_{i=1}^\ell\text{Fac}(M(w_i))$.  Suppose this containment is proper, and let $M(w)\in\Tcal$ be an indecomposable of minimum dimension such that $M(w)\notin\bigvee_{i=1}^\ell\text{Fac}(M(w_i))$.  Suppose first that $M(w)$ contains no proper nonzero submodule in $\Tcal$.  Then there must exist some $M(w^{\pr})\in\Tcal$ such that $M(w)$ is a proper quotient of $M(w^\prime)$ but $M(w^{\pr})$ has no proper submodule in $\Tcal$.  Choosing such an $M(w^{\pr})$ of maximal dimension, we have $w^{\pr}=w_i$ for some $i$ and $M(w)\in\text{Fac}(M(w_i))$, contrary to the assumption that $M(w)\notin\bigvee_{i=1}^\ell\text{Fac}(M(w_i))$.  Hence, $M(w)$ contains a proper nonzero submodule $M$ such that $M\in\Tcal$.  Since $\Tcal$ is quotient closed, $M(w)/M\in\Tcal$.  But $M(w)/M$ decomposes into a direct sum of indecomposables, each of smaller dimension than $M(w)$.  By the minimality hypothesis, $M(w)/M\in\bigvee_{i=1}^\ell\text{Fac}(M(w_i))$.  As $\bigvee_{i=1}^\ell\text{Fac}(M(w_i))$ is extension closed, it must contain $M(w)$, contrary to our assumption.

Next, we show that $\Tcal\setminus M(w_i)$ is in $\mathcal{BIC}(Q)$.  It is clear by (1) that $\Tcal\setminus M(w_i)$ is weakly extension closed.  Assume that it is not co-closed.  Let $M(w)$ be an indecomposable in $\Tcal\setminus M(w_i)$ of minimum dimension such that there exists an extension $0\ra M(u)\ra M(w)\ra M(w^{\pr})\ra 0$ for which $M(u)$ and $M(w^{\pr})$ are not in $\Tcal\setminus M(w_i)$.  Since $\Tcal$ is quotient closed, we deduce $M(w^{\pr})=M(w_i)$.  By (2), there exists some $M(u^{\pr})\in\Tcal$ such that $M(u^{\pr})$ is a proper nonzero submodule of $M(w)$.  By (1) and since $\mathcal{T}$ is quotient closed, the composition $M(u^{\pr})\ra M(w)\ra M(w_i)$ must be 0.  Hence, there is an inclusion $M(u^{\pr})\ra M(u)$, which gives an exact sequence of the form
$$0\ra M(u)/M(u^{\pr})\ra M(w)/M(u^{\pr})\ra M(w_i)\ra 0.$$
Since $\Tcal$ is extension closed and by Lemma~\ref{facmw} $a)$, $M(u)/M(u^{\pr})$ is not in $\Tcal$.  But, since $\Tcal$ is quotient closed, $M(w)/M(u^{\pr})$ is in $\Tcal$.

If $M(w)/M(u^{\pr})$ is an indecomposable, then so is $M(u)/M(u^{\pr})$, and we obtain a contradiction to the minimality of $M(w)$.  Otherwise, $M(w)/M(u^{\pr})$ is a direct sum of two string modules $M(v)\oplus M(v^{\pr})$.  In this case, $w_i$ must be a substring of one of these strings, so we may assume $M(v)\ra M(w_i)$ is a quotient map.  Since there is an extension of the form $0\ra M(u)\ra M(w)\ra M(w_i)\ra 0$, the string $v$ is of the form $u^{\prime\prime} \leftarrow w_i$ for two strings $w_i$ and $u^{\pr\pr}$ where $0\ra M(u^{\pr\pr})\ra M(v)\ra M(w_i)\ra 0$ is exact.  But this implies $M(u)/M(u^{\pr})\cong M(u^{\pr\pr})\oplus M(v^{\pr})$, so $M(u^{\pr\pr})\notin\Tcal$ while $M(v)\in\Tcal$.  Again, this contradicts the minimality of $M(w)$.  Hence, we conclude that $\Tcal\setminus M(w_i)$ is in $\mathcal{BIC}(Q)$.

Now suppose $\Tcal=\bigvee_{j=1}^m\text{Fac}(M(w_j^{\pr}))$ is some other join-representation of $\Tcal$.  For a given $i\in[\ell]$, if none of the factors $\text{Fac}(M(w_j^{\pr}))$ contains $M(w_i)$, then $\bigvee_{j=1}^m\text{Fac}(M(w_j^{\pr}))\subseteq\Tcal\setminus M(w_i)$, in contradiction with our assumption.  Hence, for all $i\in[\ell]$, there exists $j\in[m]$ such that $\text{Fac}(M(w_i))\subseteq\text{Fac}(M(w_j^{\pr}))$.  This means that our join-representation $\Tcal=\bigvee_{i=1}^\ell\text{Fac}(M(w_i))$ is canonical.
\end{proof}

Dually, every torsion class has a canonical meet-representation.

\begin{corollary}\label{canonmeetrepn}
Let $\Tcal\in\text{tors}(\Lambda)$.  Let $M(w_1),\ldots,M(w_\ell)$ be $\Lambda$-modules such that $D(\Tcal^{\perp})=\bigvee_{j=1}^\ell \text{Fac}(DM(w_i))$ is a canonical join-representation of $D(\Tcal^{\perp})$.  Then $\Tcal=\bigwedge_{i=1}^\ell{}^{\perp}\text{Sub}(M(w_i))$ is a canonical meet-representation of $\Tcal$.
\end{corollary}
\begin{proof}
We first show that $\mathcal{T} = \bigwedge_{i=1}^\ell{}^{\perp}\text{Sub}(M(w_i)).$ Observe that $$\begin{array}{rclll}
\mathcal{T} & = & {}^\perp(D(D(\mathcal{T}^\perp))) \\
& = & {}^\perp\left(D\left(\bigvee_{j=1}^\ell \text{Fac}(DM(w_i))\right)\right) \\
& = & {}^\perp\left(D\left(\bigvee_{j=1}^\ell D\text{Sub}(M(w_i))\right)\right) \\
& = & {}^\perp\left(\left(\bigvee_{j=1}^\ell DD\text{Sub}(M(w_i))\right)\right) \\
& = & {}^\perp\left(\left(\bigcap_{j=1}^\ell {}^\perp\text{Sub}(M(w_i))\right)^\perp\right) & (\text{by Propostion~\ref{meetandjointors}} \ b) \text{ }) \\
& = & \bigwedge_{j = 1}^\ell{}^\perp\text{Sub}(M(w_i)) & (\text{by Propositions~\ref{torsbij} and \ref{meetandjointors}} \ a) \text{ }).
\end{array}$$ Since the functor $D((-)^\perp): \text{tors}(\Lambda) \to \text{tors}(\Lambda^{\op})$ is an anti-isomorphism by Lemma~\ref{standarddual} and $\bigvee_{j=1}^\ell \text{Fac}(DM(w_i))$ is a canonical join-representation of $D(\mathcal{T}^\perp)$, we have by Lemma~\ref{canonjoincanonmeet} that $\bigwedge_{i=1}^\ell{}^{\perp}\text{Sub}(M(w_i))$ is a canonical meet-representation of $\Tcal.$
\end{proof}

\section{Some additional lemmas}\label{additionallemmas}

In this section, unless otherwise stated, we let $Q$ be a type $\mathbb{A}$ quiver and let $\Lambda = \Bbbk Q/{I}$ denote the cluster-tilted algebra corresonding to $Q$.

\begin{lemma}\label{intcomponents}
Let $M(u), M(v) \in \text{ind}(\Lambda\text{-mod})$ with $\text{supp}(M(u)) \cap \text{supp}(M(v)) \neq \emptyset$. Then there is a unique string $w = x_1 \leftrightarrow x_2 \cdots x_{k-1}\leftrightarrow x_k$ in $\Lambda$ such that $\text{supp}(M(u)) \cap \text{supp}(M(v)) = \{x_i\}_{i \in [k]}.$
\end{lemma}
\begin{proof}
This lemma is a consequence of \cite[Lemma 3.3]{ccs1} and \cite[Theorem 4.4]{ccs1}.
\end{proof}

\begin{lemma}\label{injsurj1dim}
{Let $M(u),M(v) \in \text{ind}(\Lambda\text{-mod})$. If $M(u) \hookrightarrow M(v)$ or $M(u) \twoheadrightarrow M(v)$, then $$\dim_\Bbbk\Hom_{\Lambda}(M(u), M(v)) = 1.$$}
\end{lemma}
\begin{proof}
This lemma is a consequence of \cite[Lemma 3.4]{ccs1}.
\end{proof}

\begin{lemma}\label{dimhom}
Let $M(u),M(v) \in \text{ind}(\Lambda\text{-mod}).$ Then $\dim_\Bbbk\Hom_\Lambda(M(u),M(v)) \le 1.$  Additionally, assume $M(u)$ is not a submodule of $M(v)$ and $M(u)$ does not surject onto $M(v)$, but that $\text{Hom}_{\Lambda_T}(M(u),M(v)) \neq 0$. Then there exists a string $w$ in $\Lambda_T$ distinct from both $u$ and $v$ such that $M(u) \twoheadrightarrow M(w) \hookrightarrow M(v).$
\end{lemma}
\begin{proof}
We can assume that $\Hom_\Lambda(M(u),M(v)) \neq 0.$ Thus, by Lemma~\ref{intcomponents} there exists a unique string $$w = x_1 \stackrel{\alpha_1}{\longleftrightarrow} x_2 \stackrel{\alpha_2}{\longleftrightarrow} \cdots \stackrel{\alpha_{m}}{\longleftrightarrow} x_{m+1}$$
that is a substring of both $u$ and $v$ such that $\pi: M(u) \twoheadrightarrow M(w)$ and $ \iota: M(w) \hookrightarrow M(v)$. It is easy to see that any map $\theta:M(u) \to M(v) $ factors as $\theta = c\iota\pi$ where $c \in \Bbbk.$ Combining this with Lemma~\ref{injsurj1dim}, we have that $\dim_\Bbbk\Hom_\Lambda(M(u),M(v)) = 1.$
\end{proof}

\begin{lemma}\label{dimext}
Assume $Q$ is of type $\mathbb{A}$ or of the form $Q = Q(n)$ and let $\Lambda = \Bbbk Q/I$ denote the corresponding cluster-tilted algebra. Let $M(u),M(v) \in \text{ind}(\Lambda\text{-mod}).$ Then $\dim_\Bbbk\Ext^1_\Lambda(M(u),M(v)) \le 1.$ 
\end{lemma}
\begin{proof}
By the Auslander-Reiten Formula (see \cite{ass06}), we have that $$\begin{array}{cclcccccc}
\dim_{\Bbbk}\Ext^1_\Lambda(M(u),M(v)) & = & \dim_{\Bbbk}\underline{\Hom}_\Lambda(\tau^{-1}M(v),M(u))\\
& \le & \dim_\Bbbk\Hom_\Lambda(\tau^{-1}M(v),M(u)) \\
& \le & 1
\end{array}$$

\noindent where the last inequality follows from Lemma~\ref{dimhom} if $Q$ is of type $\mathbb{A}$ and from Lemma~\ref{homQn} if $Q = Q(n)$ and the fact that $\tau^{-1}M(v)$ is either zero or indecomposable.
\end{proof}

\bibliographystyle{plain}
\bibliography{bib_torsion.bib}

\end{document}